\documentclass[12pt]{article}

\usepackage[margin=1in]{geometry}
\usepackage{amsmath,amssymb,amsfonts,amsthm,mathtools}
\usepackage{enumitem}

\usepackage[authoryear,round]{natbib}

\usepackage[
colorlinks=true,
citecolor=blue,
linkcolor=blue,
urlcolor=blue,
pagebackref=true
]{hyperref}

\newtheorem{theorem}{Theorem}
\newtheorem{lemma}{Lemma}

\newtheorem{assumption}{Assumption}
\newtheorem{remark}{Remark}

\newcommand{\GBS}{\mathrm{GBS}}

\newcommand{\FDP}{\operatorname{FDP}}
\newcommand{\FNP}{\operatorname{FNP}}

\newcommand{\sn}{s_n}

\newcommand{\eps}{\varepsilon}

\title{Sharp Asymptotic Minimaxity of the Gavrilov--Benjamini--Sarkar Step-Down Testing Procedure in Sparse Gaussian Sequence Models}
\author{
	Prasenjit Ghosh\thanks{Email: \texttt{prasenjit@stat.tamu.edu}}\\[0.2ex]
	{\footnotesize Department of Statistics, Texas A\&M University, College Station, TX, USA}
}
\date{}

\begin{document}
\maketitle

\begin{abstract}
False discovery rate control does not by itself guarantee optimal recovery of sparse signals: a procedure may make few false discoveries while missing an unnecessarily large fraction of the alternatives. We investigate whether the classical Gavrilov--Benjamini--Sarkar (GBS) step-down procedure also achieves the sharp frequentist minimax limits for sparse recovery in independent Gaussian sequence models. Under a beta-min separation at the critical sparse detection boundary, GBS attains the exact asymptotic risk constant $\bar\Phi(b)$ for normalized Hamming loss and combined $\mathrm{FDP}+\mathrm{FNP}$ loss. The Hamming conclusion remains valid when the number of signals is only bounded above by the nominal sparsity $s_n$. For heterogeneous signal magnitudes, the corresponding risk under either loss equals the sharp benchmark $\Lambda_n(a)+o(1)$, without assuming that $\Lambda_n(a)$ converges; this is additive, rather than necessarily relative, sharpness. We assume $s_n\to\infty$, $n/s_n\to\infty$, and a vanishing nominal level $\alpha_n$ satisfying $\log(1/\alpha_n)=o(\sqrt{\log(n/s_n)})$. Unlike step-up rules, a step-down procedure must satisfy every preceding ordered $p$-value constraint before reaching a given rank. Our analysis handles this sequential obstruction directly: a finite-sample shifted order-statistic bound makes false positives negligible, while new uniform crossing arguments show that the rejection path reaches the optimal fraction of signals. The results extend the range of testing architectures known to attain the sharp minimax benchmarks, without requiring localization of a single rejection threshold.
\end{abstract}
\bigskip


\medskip


\section{Introduction}

Sparse Gaussian sequence models provide a canonical framework for studying the
fundamental limits of large-scale multiple testing. In this setting, one observes
independent Gaussian random variables
\begin{equation}\notag
	X_i=\theta_i+Z_i,
	\qquad
	Z_i\stackrel{\mathrm{iid}}{\sim}N(0,1),
	\qquad
	i=1,\ldots,n,
	\label{eq:intro_gaussian_model}
\end{equation}
and seeks to simultaneously test
\begin{equation}\notag
	H_{0i}:\theta_i=0
	\qquad\text{versus}\qquad
	H_{1i}:\theta_i\ne0,
	\qquad
	i=1,\ldots,n.
	\label{eq:intro_testing_problem}
\end{equation}
A central feature of modern high-dimensional inference is sparsity, whereby only a
small fraction of the coordinates correspond to genuine signals. Sparse multiple
testing problems arise naturally in a wide variety of scientific applications, including
genomics, neuroimaging, signal processing, astronomy, and large-scale screening
studies. Over the past three decades, the development of principled procedures for
identifying sparse signals while controlling erroneous discoveries has become a major
theme in statistical methodology.

Among the most influential developments in this area is the false discovery rate (FDR) framework introduced by \citet{BH1995}. Since its introduction, the FDR paradigm has fundamentally reshaped the theory and practice of large-scale multiple testing by providing a principled balance between false discoveries and power. A vast literature has subsequently developed around the construction, analysis, and refinement of FDR-controlling procedures; see, for example, \citet{BY2001}, \citet{STO_2002}, \citet{STS_2004}, \citet{SKS2006}, \citet{BKY2006}, \citet{SKS2008}, \citet{GBS2009}, \citet{BLROQ2009}, \citet{SUN_CAI_2009}, \citet{NR2012}, \citet{GHS2014}, and the references therein. These developments have produced a rich collection of step-up, step-down, adaptive, empirical Bayes, and local false-discovery-rate procedures possessing attractive finite-sample error-control properties and strong empirical performance in high-dimensional settings. Beyond methodological advances, this literature has also stimulated extensive theoretical investigations into the operating characteristics, optimality properties, and large-sample behavior of multiple testing procedures under a variety of sparsity regimes and inferential loss functions.

Beyond error-rate control, a central objective in modern multiple testing theory is to understand the extent to which a given procedure approaches the fundamental limits of sparse signal recovery. While false discovery rate control provides protection against excessive false rejections, it does not, by itself, quantify how efficiently a procedure recovers sparse signals. In particular, two procedures may control the FDR at the same nominal level while exhibiting substantially different false-negative behavior, classification accuracy, and overall inferential efficiency. Consequently, an important objective of modern multiple testing theory is to determine whether procedures originally developed from error-control considerations also attain fundamental decision-theoretic optimality benchmarks for sparse signal recovery. Such investigations seek to understand the extent to which a given procedure approaches the intrinsic limits of sparse inference and to identify the mechanisms responsible for such optimality. Within this framework, asymptotic Bayes optimality under sparsity (ABOS) and sharp asymptotic minimaxity provide two complementary notions of optimal performance, arising respectively from Bayesian and frequentist decision-theoretic perspectives. A substantial body of work has investigated ABOS, minimax risk theory, posterior contraction, and sharp sparse detection boundaries; see, among others, \citet{ABDJ2006}, \citet{BCFG2011}, \citet{DG2013}, \citet{GTGC2015}, \citet{GC2016}, \citet{PaulChakrabarti2025AISM}, and the references therein. Collectively, these investigations seek to determine whether widely used multiple testing procedures attain the best possible asymptotic performance in sparse settings and to characterize the fundamental limits of sparse inference.

A natural question arising from this body of work is whether classical multiple testing procedures originally developed for error-rate control also attain these more refined optimality benchmarks. Establishing such results not only provides theoretical justification beyond nominal error guarantees, but also clarifies the extent to which a procedure approaches the fundamental limits of sparse signal recovery. From this perspective, asymptotic Bayes optimality and sharp asymptotic minimaxity furnish two complementary frameworks for evaluating the performance of multiple testing procedures under sparsity. The former assesses performance relative to an optimal Bayes benchmark, whereas the latter characterizes performance relative to the minimax risk over suitably defined sparse parameter spaces.

The present paper is concerned with the Gavrilov--Benjamini--Sarkar (GBS) step-down multiple testing procedure \citep{GBS2009}. The GBS procedure constitutes one of the most important FDR-controlling step-down methods in the multiple testing literature and has received considerable attention from both theoretical and applied perspectives. From a structural standpoint, GBS occupies a distinctive position among classical FDR-controlling procedures. Unlike the Benjamini--Hochberg step-up rule, whose asymptotic behavior has been extensively investigated through threshold-localization arguments, GBS employs a fundamentally different step-down geometry. Understanding whether such a procedure also attains optimal sparse-recovery performance is therefore of intrinsic interest and provides insight into the robustness of sparse multiple testing optimality across different testing architectures. More recently, \citet{Ghosh2026GBSABOS} established asymptotic Bayes optimality under sparsity of the GBS procedure in sparse Gaussian sequence models. That work provided a detailed characterization of the asymptotic Type-I and Type-II error behavior of GBS and demonstrated that, under suitable sparsity assumptions, the induced Bayes risk attains the optimal asymptotic benchmark.

A related but distinct line of research concerns asymptotic minimaxity and sharp risk
optimality in sparse multiple testing problems. While false discovery rate control
provides protection against excessive false discoveries, it does not directly quantify
the overall decision-theoretic performance of a testing procedure. In particular,
procedures possessing similar FDR behavior may exhibit markedly different false-negative
rates, classification accuracy, and overall risk. This observation has motivated the
study of loss functions that combine type~I and type~II errors and has led to a growing
literature on minimax optimality in sparse multiple testing; see, for example,
\citet{fromont2016family}, \citet{arias2017distribution},
\citet{belitser2020needles}, and \citet{rabinovich2020optimal}.

A natural approach to assessing such decision-theoretic performance is through minimax
risk analysis. In this framework, a testing procedure is evaluated through its
worst-case risk over a suitably defined parameter space, thereby providing guarantees
that hold uniformly across a broad class of sparse configurations. Meaningful minimax
analysis in sparse multiple testing requires both a sufficiently rich parameter space
and signal-strength conditions ensuring that non-null effects can be distinguished from
noise with nontrivial probability. These considerations have led to the development of
beta-min frameworks, which now play a central role in the study of sparse testing
boundaries and minimax optimality.

In the context of sparse Gaussian sequence models, \citet{arias2017distribution}
introduced a beta-min framework for the combined loss $\mathrm{FDP}+\mathrm{FNP}$ and
established a sharp dichotomy between impossible and possible testing regimes. Under an
equal-signal-strength assumption, they showed that below a universal logarithmic
threshold no thresholding procedure can achieve nontrivial minimax performance, whereas
above this threshold suitably calibrated BH-type procedures attain vanishing minimax
risk. Subsequently, \citet{rabinovich2020optimal} derived non-asymptotic bounds for
combined type~I and type~II error criteria in generalized Gaussian sequence models and
showed that BH-type procedures can attain the corresponding lower bounds under stronger
signal-strength assumptions. Collectively, these works highlighted the fundamental role
of beta-min separation conditions in determining the limits of reliable sparse signal
recovery.

More recently, \citet{ACR2024} developed a sharp asymptotic minimaxity theory for sparse multiple testing under natural sparse testing losses. Their work provides one of the most refined decision-theoretic characterizations currently available for sparse multiple testing by identifying the exact asymptotic risk boundaries governing sparse signal recovery and by establishing multiple testing procedures that attain these boundaries adaptively. A central component of their theory is the classical beta-min framework, which imposes a separation condition requiring every nonzero mean to exceed a critical signal-strength threshold of the form
\begin{equation}\notag
	a_{n,b}
	=
	\sqrt{2\log(n/s_n)}+b,
	\qquad
	b\in\mathbb R,
	\end{equation}
where $s_n$ denotes the number of signals. The constant $b$ determines the position of the signal strength relative to the sparse detection boundary and governs the limiting difficulty of the testing problem. The beta-min boundary plays a role analogous to a phase-transition threshold in sparse inference. Signals substantially above this boundary can typically be recovered with asymptotically vanishing error, whereas signals below the boundary become increasingly difficult to distinguish from noise. Consequently, this regime represents the most challenging nontrivial setting for sparse multiple testing and provides a natural benchmark for studying the optimality of multiple testing procedures. Beyond this classical setting, \citet{ACR2024} further developed a substantially more general heterogeneous signal-strength framework, allowing the nonzero means to vary across coordinates while preserving the same sparse detection boundary through an appropriate aggregate characterization of the signal strengths.

Within the beta-min framework, \citet{ACR2024} showed that, under both normalized classification loss and the combined $\textrm{FDP}+\textrm{FNP}$ loss, the sharp asymptotic minimax risk over the beta-min class $\Theta_b$ is governed by
\begin{equation}\notag
	\bar{\Phi}(b)
	=
	1-\Phi(b).
	\end{equation}
Thus, the sparse multiple testing problem exhibits a remarkably sharp transition: as $b$ increases, signals become easier to recover and the minimax risk decreases, whereas as $b$ decreases, the problem approaches the impossibility regime. Sharp asymptotic minimaxity is considerably more stringent than rate-optimality, as it requires a procedure not only to achieve the correct asymptotic scaling of the risk but also to attain the exact asymptotic minimax constant governing the sparse testing problem. Their theory further establishes corresponding sharp asymptotic minimaxity results over substantially broader heterogeneous signal-strength classes, thereby providing a unified minimax framework encompassing both equal- and heterogeneous-signal-strength regimes.

The results of \citet{ACR2024} further showed that certain standard procedures, including the Benjamini--Hochberg procedure with a vanishing nominal level and empirical Bayes $\ell$-value procedures under suitable spike-and-slab priors, attain these sharp minimax boundaries. These findings are particularly noteworthy in light of earlier work establishing asymptotic optimality properties of the same procedures from Bayesian decision-theoretic perspectives. For example, Benjamini--Hochberg type procedures and multiple testing rules induced by horseshoe-type shrinkage priors are known to enjoy asymptotic Bayes optimality under sparsity \citep{BCFG2011, GTGC2015, GC2016, PaulChakrabarti2025AISM} and have subsequently been shown to attain sharp asymptotic minimaxity under suitable signal-strength conditions \citep{PaulGhoshChakrabarti2025}. This emerging connection between Bayesian and frequentist notions of optimality naturally raises the question of whether similar phenomena hold more broadly for other multiple testing procedures.

More recently, \citet{Ghosh2026GBSABOS} established asymptotic Bayes optimality under sparsity of the Gavrilov--Benjamini--Sarkar procedure. In view of the aforementioned connection between asymptotic Bayes optimality under sparsity and sharp asymptotic minimaxity, it is natural to ask whether the GBS procedure also attains the sharp asymptotic minimax benchmarks established by \citet{ACR2024} under both the classical beta-min framework and the more general heterogeneous signal-strength framework. However, corresponding sharp asymptotic minimaxity results for the GBS procedure do not appear to have been previously established. This gap is particularly noteworthy because the GBS procedure is a classical FDR-controlling method whose rejection rule possesses a genuinely step-down geometry fundamentally different from that of the Benjamini--Hochberg step-up procedure. Consequently, its asymptotic minimax behavior cannot be deduced directly from existing threshold-localization analyses developed for BH-type procedures, making a separate and intrinsically step-down analysis necessary. More broadly, the present work extends the scope of sharp asymptotic minimaxity theory beyond procedures whose rejection mechanisms can be analyzed through localization of a single implicit threshold. Instead, it demonstrates that sharp asymptotic minimaxity can also be established for genuinely step-down procedures whose rejection decisions are determined by the geometry of the ordered $p$-values rather than by a single implicit threshold.


Motivated by the foregoing discussion, we investigate the sharp asymptotic minimaxity of the GBS procedure in sparse Gaussian sequence models under both the classical beta-min framework and the more general heterogeneous signal-strength framework developed by \citet{ACR2024}. Our objective is to investigate whether the GBS procedure attains the sharp sparse multiple testing boundaries identified by \citet{ACR2024} under these two complementary formulations and, more broadly, whether the connection between asymptotic Bayes optimality under sparsity and sharp asymptotic minimaxity extends to this classical step-down multiple testing procedure. To this end, we develop a direct analysis of the GBS rejection path that exploits the distinctive geometry of the GBS critical sequence and yields sharp asymptotic minimaxity results under both normalized Hamming loss and the combined $\textrm{FDP}+\textrm{FNP}$ loss.

The statistical question is not whether GBS can control false discoveries---this is already known under independence---but whether such control comes at an asymptotic cost in missed signals. Near the sparse detection boundary, a signal with magnitude $a_n^*+b$ is separated from the effective noise threshold by only a constant $b$, so the Gaussian tail $\bar\Phi(b)$ describes the irreducible missed-signal contribution. For a step-down rule, however, this heuristic faces an additional obstruction: a failure at any earlier ordered $p$-value prevents all later rejections, even if many later signals are individually strong. The substantive issue is whether this cumulative stopping constraint changes the first-order recovery limit. Our results show that it does not for GBS under the stated independent Gaussian regimes. This is an attainability result for an existing minimax boundary, not a claim to discover a new boundary or a universal theorem for all step-down procedures.

The main contributions of this paper are as follows. First, we establish that the classical Gavrilov--Benjamini--Sarkar step-down procedure attains the sharp sparse multiple testing boundaries identified by \citet{ACR2024} under the classical beta-min framework. Specifically, we prove sharp asymptotic minimaxity under normalized Hamming loss over the exact-support beta-min classes $\Theta_b$, showing that the normalized classification risk attains the optimal asymptotic constant $\bar{\Phi}(b)$. We separately prove that the same sharp constant holds over the enlarged at-most-$s_n$ beta-min class, even though the actual number of signals may vary from zero to $s_n$. The latter theorem requires a support-size-uniform crossing argument and does not follow merely by enlarging the parameter space. We then extend this sharp minimaxity theory to the heterogeneous signal-strength classes $\Theta(a,s_n)$ introduced by \citet{ACR2024}, combining a uniform normalized Hamming risk upper bound with their matching minimax lower bound. Furthermore, we establish sharp asymptotic minimaxity under the combined $\mathrm{FDP}+\mathrm{FNP}$ loss under both the classical beta-min and heterogeneous signal-strength frameworks, showing that the GBS procedure achieves the optimal asymptotic trade-off between false discoveries and missed discoveries under two of the principal sparse multiple testing criteria.

From a methodological perspective, our analysis develops a direct proof tailored to the geometry of the GBS step-down procedure. The argument is based on a shifted order-statistic inequality for the GBS critical sequence together with deterministic and empirical signal-crossing analyses for ordered alternative $p$-values. Unlike existing analyses of the Benjamini--Hochberg procedure and empirical Bayes local false-discovery-rate methods, our approach does not rely on threshold-localization arguments and is intrinsic to the GBS rejection mechanism itself. The resulting proof shows that the published minimax benchmarks are attainable by GBS even though its rejection rule cannot be characterized by a single final-rank crossing.

Finally, the present work complements two recent theoretical developments concerning the GBS procedure. In \citet{GhoshChakrabartiAdmissibility2026}, a broad class of residual-based step-down multiple testing procedures was shown to be admissible under a vector-valued loss for arbitrary covariance structures. Under independence, the classical GBS procedure is recovered as a special case of this general residual-based framework. More recently, \citet{Ghosh2026GBSABOS} established that, under independence, the GBS procedure attains asymptotic Bayes optimality under sparsity by asymptotically matching the Bayes Oracle risk. The present paper complements these results by establishing sharp asymptotic minimaxity relative to the frequentist minimax benchmark under both the classical beta-min framework and the more general heterogeneous signal-strength framework introduced by \citet{ACR2024}. Taken together, these works show that, under independence, the GBS procedure enjoys finite-sample admissibility, asymptotic Bayes optimality under sparsity, and sharp asymptotic minimaxity from three complementary decision-theoretic perspectives. These results provide complementary finite-sample, Bayesian, and frequentist perspectives on GBS, each under its own loss formulation and assumptions.

The remainder of the paper is organized as follows. Section~\ref{sec:model_framework} introduces the sparse Gaussian sequence model, the beta-min and heterogeneous signal-strength parameter spaces, the associated sharp asymptotic minimax framework, and the Gavrilov--Benjamini--Sarkar (GBS) step-down procedure. Section~\ref{sec:gbs_sharp_minimaxity_theory} develops the sharp asymptotic minimax theory of the GBS procedure. We first establish sharp asymptotic minimaxity under normalized Hamming loss for exact-support and at-most-$s_n$ beta-min classes, followed by the combined $\mathrm{FDP}+\mathrm{FNP}$ loss over the exact-support beta-min class. We then extend the normalized Hamming risk analysis to the more general heterogeneous signal-strength framework introduced by \citet{ACR2024}, establishing sharp asymptotic minimaxity under normalized Hamming loss, and further establish sharp asymptotic minimaxity under the combined $\mathrm{FDP}+\mathrm{FNP}$ loss over the heterogeneous signal-strength classes. Section~\ref{sec:discussion} concludes with a discussion of the results, their methodological implications, and directions for future research. Finally, the Appendix contains the technical proofs of the principal theoretical results.

\subsection*{Notation}
Throughout the paper, for two sequences $\{a_n\}$ and $\{b_n\}$ with
$b_n\neq0$, we write $a_n=o(b_n)$ if
\[
\lim_{n\to\infty}\frac{a_n}{b_n}=0.
\]
We write $a_n\lesssim b_n$ if there exists a constant $C>0$ such that
$a_n\le Cb_n$ for all sufficiently large $n$, and
$a_n\asymp b_n$ if both $a_n\lesssim b_n$ and $b_n\lesssim a_n$ hold.

For any two real numbers $a$ and $b$, we write
\[
a \vee b = \max\{a,b\}, \qquad
a \wedge b = \min\{a,b\},
\]
and, for any real number $x$,
\[
(x)_+ = \max\{x,0\}.
\]
We let $\phi(\cdot)$ and $\Phi(\cdot)$ denote the probability density
function and cumulative distribution function of the standard normal
distribution, respectively, and write
\[
\bar{\Phi}(x)=1-\Phi(x),
\qquad x\in\mathbb{R}.
\]

\section{Sparse Gaussian Model, Sharp Asymptotic Minimax Framework, and the GBS Procedure}
\label{sec:model_framework}

In this section, we introduce the mathematical framework underlying the subsequent theoretical developments. We begin by describing the sparse Gaussian sequence model and the associated beta-min parameter spaces introduced by \citet{ACR2024}. We then review the sharp asymptotic minimax framework for sparse multiple testing, including the normalized Hamming loss and the combined $\textrm{FDP}+\textrm{FNP}$ loss, together with the corresponding optimal minimax benchmarks established by \citet{ACR2024}. Finally, we introduce the Gavrilov--Benjamini--Sarkar (GBS) step-down procedure \citep{GBS2009}, which is the primary object of study in the present paper.

\subsection{Sparse Gaussian Sequence Model}
\label{sec:gaussian_seq}

We consider the sparse Gaussian sequence model
\begin{equation}
	X_i
	=
	\theta_i+\varepsilon_i,
	\qquad
	\varepsilon_i
	\stackrel{\mathrm{i.i.d.}}{\sim}
	N(0,1),
	\qquad
	i=1,\ldots,n,
	\label{eq:model}
\end{equation}
where $\theta=(\theta_1,\ldots,\theta_n)\in\mathbb R^n$ is an unknown mean vector. For each coordinate $i$, our objective is to simultaneously test
\[
H_{0i}:\theta_i=0,
\qquad
\text{versus}
\qquad
H_{1i}:\theta_i\neq0,
\qquad
i=1,\ldots,n.
\]

Let
\[
S_\theta
=
\{i:\theta_i\neq0\}
\]
denote the support of $\theta$, and write
\[
k(\theta):=|S_\theta|
\]
for its actual support size. Throughout the paper, $s_n$ is a deterministic nominal sparsity parameter: $k(\theta)=s_n$ on the exact-support classes, while $0\le k(\theta)\le s_n$ on the at-most-sparsity class. In particular, the normalization in the at-most-sparsity Hamming risk is always $s_n$, not $k(\theta)$. We work under the following asymptotic sparsity regime.

\begin{assumption}[Sparse asymptotic regime]
	\label{ass:sparse}
	Throughout the paper,
	\[
	n\to\infty,
	\qquad
	s_n\to\infty,
	\qquad
	\frac{n}{s_n}\to\infty.
	\]
\end{assumption}

The primary objective in the sparse Gaussian sequence model is to identify the unknown support $S_\theta$ while controlling erroneous discoveries. Throughout this paper, we work in the high-dimensional sparse regime described by Assumption~\ref{ass:sparse}, under which the number of signals diverges but remains negligible relative to the ambient dimension. This asymptotic framework forms the basis of the sharp asymptotic minimax theory for sparse multiple testing developed by \citet{ACR2024} and will be adopted throughout the sequel.

\subsection{Beta-Min Parameter Classes and Sharp Asymptotic Minimaxity}
\label{subsec:beta_min_framework}

Following \citet{ACR2024}, we evaluate multiple testing procedures over the beta-min parameter spaces. For a fixed constant $b\in\mathbb R$, define
\begin{equation}
	a_{n,b}
	=
	\sqrt{2\log(n/s_n)}
	+
	b.
	\label{eq:beta_min_boundary}
\end{equation}
The corresponding beta-min parameter class is
\begin{equation}
	\Theta_b
	=
	\left\{
	\theta\in\mathbb R^n:
	|S_\theta|=s_n,
	\quad
	|\theta_i|
	\ge
	a_{n,b}
	\ \text{for all }i\in S_\theta
	\right\}.
	\label{eq:theta_b}
\end{equation}

The parameter spaces $\Theta_b$ characterize the critical signal-strength regime for sparse multiple testing. The threshold $a_{n,b}$ represents a refinement of the classical sparse detection boundary, with the constant $b$ determining the position of the signal strength relative to this boundary. Signals substantially above the beta-min threshold can typically be recovered with asymptotically vanishing error, whereas signals below the threshold become increasingly difficult to distinguish from noise, rendering reliable recovery impossible in general. Consequently, the beta-min classes $\Theta_b$ represent the most challenging nontrivial regime for sparse multiple testing and provide the natural parameter spaces for studying sharp asymptotic minimaxity. Throughout this paper, they serve as the benchmark parameter spaces over which the minimax performance of multiple testing procedures is evaluated.

Let $\varphi:\mathbb R^n\to\{0,1\}^n$ be a measurable multiple testing rule, with $\varphi(X)=(\varphi_1(X),\ldots,\varphi_n(X))$. Here $\varphi_i(X)=1$ indicates rejection of $H_{0i}$ and $\varphi_i(X)=0$ otherwise. Throughout, infima over $\varphi$ range over all such measurable rules, as in \citet{ACR2024}. We suppress the argument $X$ when no ambiguity arises. Define the numbers of false positives and false negatives by
\begin{equation}
	V(\theta,\varphi)
	=
	\sum_{i:\theta_i=0}\varphi_i,
	\qquad
	T(\theta,\varphi)
	=
	\sum_{i:\theta_i\neq0}(1-\varphi_i).
	\label{eq:V_T_definition}
\end{equation}

Following \citet{ACR2024}, we evaluate multiple testing procedures under two complementary sparse testing loss functions. The first is the normalized Hamming loss, which measures the overall classification accuracy by combining false positives and false negatives. The second is the combined false discovery proportion (FDP) and false nondiscovery proportion (FNP) loss, which balances the proportions of erroneous rejections and missed discoveries. Together, these two loss functions provide complementary decision-theoretic measures of performance and form the basis of the sharp asymptotic minimax framework considered in this paper.

The normalized Hamming loss is defined by
\begin{equation}
	\mathcal{L}_H(\theta,\varphi)
	=
	\frac{1}{s_n}
	\left\{
	V(\theta,\varphi)
	+
	T(\theta,\varphi)
	\right\},
	\label{eq:hamming_loss}
\end{equation}
%
with corresponding minimax risk
\begin{equation}
	\mathcal R_H(\Theta_b)
	=
	\inf_{\varphi}
	\sup_{\theta\in\Theta_b}
 	\mathbb E_\theta 	\mathcal{L}_H(\theta,\varphi)
	\label{eq:minimax_hamming}
\end{equation}


Following \citet{ACR2024}, we also consider the combined false discovery
proportion (FDP) and false nondiscovery proportion (FNP) loss. For convenience,
we use FNP for the support-normalized proportion of missed signals. The false
discovery proportion and false nondiscovery proportion are respectively defined
by
\begin{equation}
	\FDP(\theta,\varphi)
	=
	\frac{V(\theta,\varphi)}
	{1\vee\sum_{i=1}^n\varphi_i},
	\qquad
	\FNP(\theta,\varphi)
	=
	\frac{T(\theta,\varphi)}
	{1\vee s_n},
	\label{eq:FDP_FNP}
\end{equation}
On the exact-support classes, $|S_\theta|=s_n$, so that
$\mathbb E_\theta\{\FNP(\theta,\varphi)\}=\mathbb E_\theta T(\theta,\varphi)/s_n$
coincides with the false negative rate (FNR) of \citet{ACR2024}.
This support-based denominator should not be confused with the alternative
convention that divides missed signals by the number of nonrejections.

The composite $FDP+FNP$ loss is defined as
\begin{equation}
	\mathcal{L}_{C}(\theta,\varphi)
	=
	\FDP(\theta,\varphi)
	+
	\FNP(\theta,\varphi),
	\label{eq:FDP_FNP_loss}
\end{equation}
with associated minimax risk
\begin{equation}
	\mathcal{R}_{C}(\Theta_b)
	=
	\inf_{\varphi}
	\sup_{\theta\in\Theta_b}
	\mathbb{E}_{\theta} \mathcal{L}_{C}(\theta,\varphi).
	\label{eq:minimax_FDP_FNP}
\end{equation}

A fundamental result of \citet{ACR2024} establishes the sharp asymptotic
minimax benchmarks under both normalized Hamming loss and the combined
$\mathrm{FDP}+\mathrm{FNP}$ loss. Specifically,
\begin{equation}
	\mathcal R_H(\Theta_b)
	=
	\bar{\Phi}(b)+o(1),
	\qquad
	\mathcal{R}_{C}(\Theta_b)
	=
	\bar{\Phi}(b)+o(1),
	\label{eq:sharp_boundary}
\end{equation}
where $\bar{\Phi}(b)=1-\Phi(b)$. Thus, the optimal asymptotic performance under both loss functions is governed by the same universal constant
$\bar{\Phi}(b)$. Moreover, \citet{ACR2024} showed that these minimax benchmarks are attained adaptively by suitably calibrated Benjamini--Hochberg procedures and empirical Bayes $\ell$-value procedures induced by spike-and-slab priors. Consequently, their work provides the precise decision-theoretic benchmarks
for sparse multiple testing under both normalized Hamming loss and the combined $\mathrm{FDP}+\mathrm{FNP}$ loss. Therefore, establishing sharp asymptotic minimaxity of a multiple testing procedure amounts to showing that
its maximal risk over the beta-min class $\Theta_b$ asymptotically coincides with the corresponding minimax benchmark under each of these loss functions. The principal objective of the present paper is to establish this property for
the classical Gavrilov--Benjamini--Sarkar (GBS) step-down procedure \citep{GBS2009}. We therefore conclude this section by briefly reviewing the
GBS procedure and introducing the notation that will be used throughout the remainder of the paper.

\subsection{Heterogeneous Signal-Strength Framework}
\label{subsec:heterogeneous_framework}

The beta-min framework described in the previous subsection provides a canonical local formulation of sparse multiple testing by requiring all nonzero means to exceed a common signal-strength threshold. This setting has played a central role in the development of sharp asymptotic minimaxity theory because it captures the most challenging nontrivial regime for sparse signal recovery through a single separation parameter. Nevertheless, in many high-dimensional applications the nonzero signals are inherently heterogeneous and may exhibit substantially different magnitudes across coordinates. Consequently, restricting attention to a common signal-strength threshold does not fully capture the intrinsic difficulty of sparse multiple testing.

To address this limitation, \citet{ACR2024} introduced a considerably more general heterogeneous signal-strength framework that permits the nonzero means to vary arbitrarily while preserving the same sparse asymptotic regime. Rather than being characterized by a single separation parameter, the asymptotic difficulty of the testing problem is now determined collectively by the entire configuration of signal strengths. This leads to a substantially richer sharp asymptotic minimaxity theory that contains the classical beta-min formulation as a special case.

Specifically, let
\[
a=(a_1,\ldots,a_{s_n})
\]
be a vector of positive numbers, and define the heterogeneous signal-strength parameter class
\begin{equation}
	\Theta(a,s_n)
	=
	\left\{
	\theta\in\mathbb R^n:
	|S_\theta|=s_n,\ 
	\text{there exist distinct } i_1,\ldots,i_{s_n}
	\text{ such that }
	|\theta_{i_j}|\ge a_j,\ 
	1\le j\le s_n
	\right\}.
	\label{eq:theta_general_a}
\end{equation}

Throughout this framework, we assume
\begin{equation}
	n\to\infty,
	\qquad
	s_n\to\infty,
	\qquad
	\frac{n}{s_n}\to\infty,
	\label{eq:general_sparse_regime}
\end{equation}
and define the sparse detection boundary
\begin{equation}
	a_n^*
	=
	\sqrt{2\log(n/s_n)}.
	\label{eq:oracle_threshold_general}
\end{equation}

The sharp asymptotic minimax benchmark over the heterogeneous parameter classes is characterized by the aggregate quantity
\begin{equation}
	\Lambda_n(a)
	=
	\frac{1}{s_n}
	\sum_{j=1}^{s_n}
	\bar\Phi(a_j-a_n^*),
	\label{eq:Lambda_general}
\end{equation}
or equivalently,
\[
1-\Lambda_n(a)
=
\frac{1}{s_n}
\sum_{j=1}^{s_n}
\Phi(a_j-a_n^*).
\]

The asymptotic statements below are additive: the error is $o(1)$, not necessarily $o(\Lambda_n(a))$ when $\Lambda_n(a)\to0$.

    The quantity $\Lambda_n(a)$ summarizes the collective difficulty of recovering the heterogeneous signals and plays the role of the sharp asymptotic minimax benchmark over the parameter classes $\Theta(a,s_n)$. Unlike the beta-min framework, where the asymptotic benchmark is determined by the single constant $\bar{\Phi}(b)$, the heterogeneous framework captures the combined contribution of all signal strengths through the aggregate quantity $\Lambda_n(a)$.

The heterogeneous framework strictly generalizes the classical beta-min formulation. Indeed, when
\[
a_j=a_n^*+b,
\qquad
1\le j\le s_n,
\]
for some fixed $b\in\mathbb R$, we have
\[
\Lambda_n(a)
=
\bar\Phi(b),
\]
so that the heterogeneous signal-strength benchmark reduces exactly to the sharp asymptotic minimax constant obtained under the beta-min framework.

\subsection{The Gavrilov--Benjamini--Sarkar (GBS) Step-Down Procedure}

We now review the Gavrilov--Benjamini--Sarkar (GBS) step-down procedure
\citep{GBS2009}, which constitutes the principal object of study in the
present paper. Although originally developed as an adaptive false discovery
rate controlling procedure under independence, its structural properties are
fundamentally different from those of the classical Benjamini--Hochberg
step-up procedure and play a crucial role in the subsequent asymptotic
analysis.

For each $i=1,\ldots,n$, define the two-sided Gaussian $p$-value by
\begin{equation}
	P_i
	=
	2\bar{\Phi}(|X_i|).
	\label{eq:pvalue_definition}
\end{equation}
Let
\begin{equation}
	P_{(1)}
	\le
	\cdots
	\le
	P_{(n)}
	\label{eq:ordered_pvalues}
\end{equation}
denote the ordered $p$-values.

For a nominal significance level $\alpha_n\in(0,1)$, the GBS critical
constants are defined by
\begin{equation}
	c_j
	=
	\frac{j\alpha_n}
	{n+1-j(1-\alpha_n)},
	\qquad
	j=1,\ldots,n.
	\label{eq:gbs_critical_constants}
\end{equation}
When $j=o(n)$, these critical values satisfy
\[
c_j
=
\frac{j\alpha_n}{n}
\{1+o(1)\},
\]
so that, in the sparse regime, the early portion of the GBS critical sequence
behaves approximately linearly.

The GBS rejection number is given by
\begin{equation}
	R_n^{\GBS}
	=
	\max
	\left\{
	r:
	P_{(j)}
	\le
	c_j
	\text{ for every }
	1\le j\le r
	\right\},
	\label{eq:gbs_rejection_number}
\end{equation}
with the convention that $R_n^{\GBS}=0$ whenever the above set is empty.
The procedure rejects precisely the hypotheses
\begin{equation}
	H_{0(1)},\ldots,H_{0(R_n^{\GBS})}.
	\label{eq:gbs_rejected_hypotheses}
\end{equation}

Unlike step-up procedures, whose rejection set is determined by locating a
single threshold crossing, the GBS procedure requires the entire initial
segment of ordered $p$-values to remain below the corresponding critical
constants. Consequently, the rejection path depends simultaneously on all
preceding inequalities
\[
P_{(j)}\le c_j,
\qquad
j=1,\ldots,R_n^{\GBS},
\]
giving the GBS procedure a genuinely step-down geometry. This feature makes
its asymptotic analysis substantially different from existing analyses of
BH-type procedures based on threshold-localization arguments.

Throughout the paper, we assume that the nominal significance level satisfies
the following regularity condition.

\begin{assumption}[Nominal significance level]
	\label{ass:alpha}
	The nominal level satisfies
	\[
	\alpha_n\to0,
	\qquad
	\log(1/\alpha_n)=o\left\{\sqrt{\log(n/s_n)}\right\}.
	\]
\end{assumption}

Assumption~\ref{ass:alpha} ensures that the nominal level decreases
sufficiently slowly for the GBS critical sequence to remain informative in
the sparse asymptotic regime while simultaneously guaranteeing asymptotically
negligible false-positive contributions. This assumption will be used
throughout the subsequent theoretical developments.

The proofs developed in this paper exploit the specific geometry of the GBS
critical sequence rather than reducing the analysis to localization of an
implicit rejection threshold. Our approach is based on a shifted
order-statistic inequality together with deterministic and empirical
signal-crossing arguments, yielding a proof that is intrinsic to the GBS
step-down procedure.

\section{Sharp Asymptotic Minimaxity of the GBS Procedure}
\label{sec:gbs_sharp_minimaxity_theory}

In this section, we establish the sharp asymptotic minimaxity of the classical Gavrilov--Benjamini--Sarkar (GBS) step-down procedure under the sparse Gaussian sequence model introduced in Section~\ref{sec:model_framework}. We first consider the classical beta-min framework, including both exact-support and at-most-$s_n$ formulations, of \citet{ACR2024}, where the performance of a multiple testing procedure is evaluated through its worst-case risk over the parameter classes $\Theta_b$. We establish sharp asymptotic minimaxity of the GBS procedure under both the normalized Hamming loss and the combined $\mathrm{FDP}+\mathrm{FNP}$ loss by showing that its asymptotic risk attains the exact minimax benchmark identified by \citet{ACR2024}. We then extend these results to the more general heterogeneous signal-strength framework introduced by \citet{ACR2024}. Under the normalized Hamming loss, we establish sharp asymptotic minimaxity
over the heterogeneous signal-strength classes. We further show that, under the combined $\textrm{FDP}+\textrm{FNP}$ loss, the GBS procedure continues to attain the corresponding sharp asymptotic minimax benchmark.

\subsection{Sharp Asymptotic Minimaxity under the Beta-Min Framework}

The first main result establishes the sharp asymptotic minimaxity of the GBS procedure under normalized Hamming loss over the classical beta-min parameter classes. It shows that the normalized classification risk of the GBS procedure attains the exact sharp asymptotic minimax benchmark identified by \citet{ACR2024}. Thus, despite its fundamentally different step-down rejection mechanism, the GBS procedure achieves the same optimal asymptotic performance as previously established for suitably calibrated Benjamini--Hochberg procedures and empirical Bayes $\ell$-value methods.

\begin{theorem}[Sharp asymptotic minimaxity of GBS under normalized Hamming loss]
	\label{thm:hamming-main}
	Assume the sparse Gaussian sequence model \eqref{eq:intro_gaussian_model} and Assumptions
	\ref{ass:sparse}--\ref{ass:alpha}. Fix $b\in\mathbb R$ and consider the beta-min class $\Theta_b$ defined in \eqref{eq:theta_b}. Then the GBS procedure satisfies
	\begin{equation}
		\sup_{\theta\in\Theta_b}
		\frac{1}{s_n}
		\mathbb E_{\theta}
		\{V(\theta,\varphi^{\GBS})
		+
		T(\theta,\varphi^{\GBS})\}
		\le
		\bar\Phi(b)+o(1).
		\label{eq:hamming_upper_bound}
	\end{equation}
	Consequently, using the sharp minimax lower bound for normalized classification risk established by \citet{ACR2024},
	\begin{equation}
		\sup_{\theta\in\Theta_b}
		\frac{1}{s_n}
		\mathbb E_{\theta}
		\{V(\theta,\varphi^{\GBS})
		+
		T(\theta,\varphi^{\GBS})\}
		=
		\mathcal R_H(\Theta_b)+o(1)
		=
		\bar\Phi(b)+o(1).
		\label{eq:hamming_sharp_minimax}
	\end{equation}
	Thus the GBS procedure is sharply asymptotically minimax under normalized Hamming loss.
\end{theorem}
\begin{proof}[{\bf Proof of Theorem \ref{thm:hamming-main}}]
	Combining Lemmas~\ref{lem:false-positive} and~\ref{lem:false-negative} gives
	\begin{equation}
		\sup_{\theta\in\Theta_b}
		\frac{1}{s_n}
		\mathbb{E}_{\theta}\{V(\theta,\varphi^{\GBS})+T(\theta,\varphi^{\GBS})\}
		\le \bar\Phi(b)+o(1).
		\label{eq:hamming_upper_bound_final}
	\end{equation}
For completeness, we explain why the matching lower bound applies to the
exact-support class $\Theta_b$, rather than deducing it from a lower bound
for the larger at-most-$s_n$ class. Write
$M_n=\inf_\varphi\sup_{\theta\in\Theta_b}
\mathbb E_\theta(V+T)/s_n$. The rule rejecting no hypotheses has
risk one; hence $M_n\le1$. Choose a procedure $\varphi_n$ whose maximal
normalized Hamming risk is at most $M_n+n^{-1}\le1+n^{-1}$.
Let $R_n$ denote its number of rejections. Since the support size equals
$s_n$ and $R_n\le s_n+V$, for any $B_n>1$,
\[
\sup_{\theta\in\Theta_b}
\mathbb P_\theta(R_n>B_ns_n)
\le\frac{1+n^{-1}}{B_n-1}.
\]
Take, for example,
$B_n=\min\{\log\log(n/s_n)+2,
\exp[(\log(n/s_n))^{1/4}/2]\}$, which diverges and
satisfies $B_n^2\le\exp[(\log(n/s_n))^{1/4}]$
for all sufficiently large $n$. In the Gaussian location model, Assumption~1A of \citet{ACR2024} holds with $a_n^*=\sqrt{2\log(n/s_n)}$ and $\delta_n=(\log(n/s_n))^{-1/4}$; the Gaussian-tail calculation given in the proof of Theorem~\ref{thm:hamming_general} also verifies $B_n^2\le (n/(3s_n))\bar\Phi(a_n^*-\delta_n)$ eventually. In particular $B_n$ satisfies the published Theorem~3 growth condition and the Theorem~6 tail condition. Thus $\varphi_n$ belongs to the
sparsity-preserving class in \citet[Theorem~3]{ACR2024},
whose exact-support FNR lower bound yields
\[
M_n+n^{-1}\ge
\sup_{\theta\in\Theta_b}\mathrm{FNR}(\theta,\varphi_n)
\ge\bar\Phi(b)-o(1).
\]
Together with the GBS upper bound, this proves the assertion.
\end{proof}

The preceding result has a direct statistical interpretation. At the beta-min boundary $a_n^*=\sqrt{2\log(n/s_n)}$, a signal of magnitude $a_n^*+b$ is separated from the effective sparse noise threshold by the fixed offset $b$. A threshold placed at $a_n^*$ would miss such a signal with asymptotic probability $\bar\Phi(b)$. Theorem~\ref{thm:hamming-main} shows that GBS attains this same first-order benchmark, despite the fact that it does not apply a fixed threshold and must satisfy every preceding ordered $p$-value constraint. The substantive conclusion is not merely that false discoveries are controlled, but that premature sequential stopping does not produce an additional first-order classification penalty under the stated assumptions. The matching minimax lower bound is due to \citet{ACR2024}; the new crossing analysis establishes that the GBS rejection path attains it.

Theorem~\ref{thm:hamming-main} assumes that the support size equals $s_n$. In applications, one may instead have a reliable upper bound $s_n$ without knowing the actual number of signals. The next theorem examines whether the same sharp normalized Hamming guarantee continues to hold uniformly when the support size is allowed to vary between zero and $s_n$. Its proof requires control of the rejection path uniformly over the relevant support sizes; the auxiliary crossing statement and proof are given in the Appendix (Lemma~\ref{lem:uniform_support_transfer}).

\begin{theorem}[Sharp asymptotic minimaxity under normalized Hamming loss over at-most-$s_n$ sparse classes]
\label{thm:at-most-hamming}
Under the assumptions of Theorem~\ref{thm:hamming-main}, define
\[
a_n^*=\sqrt{2\log(n/s_n)},\qquad
\widetilde\Theta_b(s_n)
=\bigl\{\theta\in\mathbb R^n:\ \|\theta\|_0\le s_n,\quad
|\theta_i|\ge a_n^*+b\text{ if }\theta_i\ne0\bigr\}.
\]
Then
\[
\inf_\varphi\sup_{\theta\in\widetilde\Theta_b(s_n)}
\frac{\mathbb E_\theta(V+T)}{s_n}
=\sup_{\theta\in\widetilde\Theta_b(s_n)}
\frac{\mathbb E_\theta(V_{\GBS}+T_{\GBS})}{s_n}+o(1)
=\bar\Phi(b)+o(1).
\]
\end{theorem}
Unlike the exact-support setting of Theorem~\ref{thm:hamming-main}, the actual number of nonzero coordinates may vary throughout the enlarged parameter class. The sharp constant is nevertheless preserved. Establishing this extension requires a support-size-uniform signal-crossing argument, rather than a direct application of the exact-support result.
\begin{proof}[{\bf Proof of Theorem~\ref{thm:at-most-hamming}}]
The minimax lower bound follows from the inclusion
$\Theta_b\subseteq\widetilde\Theta_b(s_n)$ and the exact-support
lower bound established in Theorem~\ref{thm:hamming-main}.
We prove the matching upper bound directly, without augmenting the
support or coupling two signal configurations.

Write $a_{n,k}=\sqrt{2\log(n/k)}$ for $1\le k\le s_n$.
Fix $\eta\in(0,\bar\Phi(b))$.  When $k=\|\theta\|_0\le\eta s_n$,
the deterministic inequality $T_{\GBS}\le k$ gives
\begin{equation}
\sup_{\substack{\theta\in\widetilde\Theta_b(s_n)\\
\|\theta\|_0\le\eta s_n}}
\frac{\mathbb E_\theta T_{\GBS}}{s_n}\le\eta.
\label{eq:atmost-small-support}
\end{equation}

We next consider $\eta s_n<k\le s_n$.  Since
\begin{equation}
0\le a_{n,k}-a_n^*
=\frac{2\log(s_n/k)}{a_{n,k}+a_n^*}
\le\frac{2\log(1/\eta)}{a_n^*}=o(1),
\label{eq:atmost-boundary-shift}
\end{equation}
uniformly over these $k$, for each fixed $\zeta>0$ and all sufficiently
large $n$ every such $\theta$ satisfies
$|\theta_i|\ge a_{n,k}+b-\zeta$ on its support.
Furthermore, uniformly over the same range,
\[
k\to\infty,\qquad n/k\to\infty,\qquad
\log(1/\alpha_n)=o(a_{n,k}).
\]

For fixed $\zeta>0$ and
$\varepsilon\in(0,\Phi(b-\zeta))$, put
$r_{n,k}=\lfloor(\Phi(b-\zeta)-\varepsilon)k\rfloor$.
By Lemma~\ref{lem:uniform_support_transfer} and the preceding
beta-min boundary shift,
\begin{equation}
\sup_{\eta s_n<k\le s_n}\ \sup_{\substack{\theta\in\widetilde\Theta_b(s_n)\\
\|\theta\|_0=k}}
\mathbb P_\theta\{R_n^{\GBS}<r_{n,k}\}\longrightarrow0.
\label{eq:atmost-uniform-crossing}
\end{equation}

For any $\theta$ with $k$ signals, the identity
$T_{\GBS}=k-R_n^{\GBS}+V_{\GBS}$ implies, pathwise,
\[
T_{\GBS}\le k-r_{n,k}+V_{\GBS}
+k\mathbf1\{R_n^{\GBS}<r_{n,k}\}.
\]
The shifted order-statistic inequality, as used in
Lemma~\ref{lem:false-positive}, gives uniformly for $k\le s_n$
\[
\mathbb E_\theta V_{\GBS}
\le\frac{\alpha_n}{1-\alpha_n}(k+1),\qquad
\sup_{\theta\in\widetilde\Theta_b(s_n)}
\frac{\mathbb E_\theta V_{\GBS}}{s_n}=o(1).
\]
Combining these inequalities with~\eqref{eq:atmost-uniform-crossing},
we obtain
\begin{align*}
\limsup_{n\to\infty}
\sup_{\substack{\theta\in\widetilde\Theta_b(s_n)\\
\eta s_n<\|\theta\|_0\le s_n}}
\frac{\mathbb E_\theta T_{\GBS}}{s_n}
&\le 1-\Phi(b-\zeta)+\varepsilon
=\bar\Phi(b-\zeta)+\varepsilon.
\end{align*}
Together with~\eqref{eq:atmost-small-support}, this yields
\[
\limsup_{n\to\infty}\sup_{\theta\in\widetilde\Theta_b(s_n)}
\frac{\mathbb E_\theta T_{\GBS}}{s_n}
\le\max\{\eta,\bar\Phi(b-\zeta)+\varepsilon\}.
\]
First let $\varepsilon\downarrow0$, then $\zeta\downarrow0$.
Since $\eta<\bar\Phi(b)$, the upper bound becomes $\bar\Phi(b)$.
Adding the uniform $o(1)$ false-positive term proves the claimed
GBS upper bound, and the exact-support lower bound completes the proof.
\end{proof}

\begin{remark}[Scope of the at-most-sparsity guarantee]
Theorem~\ref{thm:at-most-hamming} permits $k(\theta)=\|\theta\|_0$ to range from zero to $s_n$, but the Hamming loss is normalized by the \emph{nominal upper bound} $s_n$, not by $k(\theta)$. The uniform result is useful when a conservative upper bound on the number of signals is available and the actual support size is uncertain. It does not imply sharpness under the unknown actual-support normalization $k(\theta)$, which would be undefined at $k=0$ without an additional convention. Nor does it imply an at-most-sparsity $\mathrm{FDP}+\mathrm{FNP}$ result: the exact-support definition of FNP and its minimax lower bound would require separate treatment.
\end{remark}

Theorems~\ref{thm:hamming-main} and~\ref{thm:at-most-hamming} constitute the foundational exact- and at-most-support Hamming results of the present paper. Together they establish that the classical Gavrilov--Benjamini--Sarkar step-down procedure attains the exact sharp asymptotic minimax benchmark identified by \citet{ACR2024} under normalized Hamming loss. Consequently, despite its fundamentally different step-down geometry and the absence of a threshold-localization argument, the GBS procedure achieves the same optimal asymptotic performance as the suitably calibrated Benjamini--Hochberg procedure and empirical Bayes $\ell$-value procedures induced by spike-and-slab priors. Taken together with the recent asymptotic Bayes optimality result of \citet{Ghosh2026GBSABOS}, Theorem~\ref{thm:hamming-main} substantially strengthens the theoretical understanding of the GBS procedure by showing that it enjoys not only asymptotic Bayes optimality under sparsity but also sharp asymptotic minimaxity at the sparse detection boundary.

\begin{remark}
Theorem~\ref{thm:hamming-main} is stated for a fixed constant
	$b\in\mathbb R$, corresponding to the critical beta-min regime where the
	signal strengths remain within a constant distance of the sparse detection
	boundary. This is precisely the regime in which the testing problem exhibits
	nontrivial asymptotic behavior and the sharp minimax risk converges to the
	nondegenerate limit $\bar{\Phi}(b)\in(0,1)$. It would be of independent
	interest to investigate the extreme regimes $b=b_n\to+\infty$ and
	$b=b_n\to-\infty$. The former corresponds to increasingly strong signals, for
	which the normalized Hamming risk is expected to vanish asymptotically,
	whereas the latter approaches the impossibility boundary, where the sharp
	minimax risk tends to one and reliable signal recovery becomes impossible.
\end{remark}	


We now turn to the combined $\textrm{FDP}+\textrm{FNP}$ loss introduced in Section~\ref{sec:model_framework}. Unlike the normalized Hamming loss, which measures the overall classification error, the $\textrm{FDP}+\textrm{FNP}$ criterion separately quantifies false discoveries and missed discoveries and therefore provides a complementary decision-theoretic assessment of multiple testing procedures. As discussed in Section~\ref{sec:model_framework}, \citet{ACR2024} established that the sharp asymptotic minimax benchmark under this loss coincides with that under normalized Hamming loss and is again given by the universal constant $\bar{\Phi}(b)$. We now show that the GBS procedure also attains this optimal benchmark.

\begin{theorem}[Sharp asymptotic minimaxity under $\textrm{FDP}+\textrm{FNP}$ loss]
	\label{thm:fdp-fnp-main}
	Assume the sparse Gaussian sequence model and Assumptions
	\ref{ass:sparse}--\ref{ass:alpha}. Fix $b\in\mathbb R$ and consider the beta-min class
	$\Theta_b$. Then the GBS procedure satisfies
	\begin{equation}
		\sup_{\theta\in\Theta_b}
		\mathbb E_\theta
		\left\{
		\mathrm{FDP}(\theta,\varphi^{\GBS})
		+
		\mathrm{FNP}(\theta,\varphi^{\GBS})
		\right\}
		\le
		\bar\Phi(b)+o(1).
		\label{eq:fdp_fnp_upper_bound}
	\end{equation}
	Consequently, using the sharp minimax lower bound of
	\citet{ACR2024}, the GBS procedure is sharply asymptotically minimax
	under $\textrm{FDP}+\textrm{FNP}$ loss.
\end{theorem}

Theorem~\ref{thm:fdp-fnp-main} complements
Theorem~\ref{thm:hamming-main} by extending the sharp asymptotic minimaxity of
the GBS procedure from normalized Hamming loss to the combined $\textrm{FDP}+\textrm{FNP}$ loss.
Since the false discovery proportion and false non-discovery proportion
measure two fundamentally different aspects of multiple testing performance,
this theorem demonstrates that the GBS procedure achieves the optimal
asymptotic trade-off between false discoveries and missed discoveries at the
sparse detection boundary. Consequently, the GBS procedure is sharply
asymptotically minimax under both principal sparse testing criteria
investigated by \citet{ACR2024}.

\begin{proof}[{\bf Proof of Theorem~\ref{thm:fdp-fnp-main}}]
	The proof follows immediately by combining the sharp false-negative bound of
	Lemma~\ref{lem:false-negative} with the false discovery rate control property
	of the GBS procedure under independence.	
	By the FDR control property of the GBS procedure under independence,
	\begin{equation}
		\sup_{\theta\in\Theta_b}
		\mathbb E_\theta
		\{\mathrm{FDP}(\theta,\varphi^{\GBS})\}
		\le
		\alpha_n.
		\label{eq:FDP_GBS_control}
	\end{equation}
	Since $\alpha_n\to0$, \eqref{eq:FDP_GBS_control} implies
	\begin{equation}
		\sup_{\theta\in\Theta_b}
		\mathbb E_\theta
		\{\mathrm{FDP}(\theta,\varphi^{\GBS})\}
		=
		o(1).
		\label{eq:FDP_negligible}
	\end{equation}
	
	Next, by Lemma~\ref{lem:false-negative},
	\begin{equation}
		\sup_{\theta\in\Theta_b}
		\mathbb E_\theta
		\{\mathrm{FNP}(\theta,\varphi^{\GBS})\}
		=
		\sup_{\theta\in\Theta_b}
		\frac{1}{s_n}
		\mathbb E_\theta T(\theta,\varphi^{\GBS})
		\le
		\bar\Phi(b)+o(1).
		\label{eq:FNP_bound_from_FN}
	\end{equation}
	Combining \eqref{eq:FDP_negligible} and \eqref{eq:FNP_bound_from_FN} gives
	\eqref{eq:fdp_fnp_upper_bound}. The matching lower bound follows from the sharp minimax
	lower bound for $\textrm{FDP}+\textrm{FNP}$ risk in \citet{ACR2024}. This proves the theorem.
\end{proof}

\subsection{Sharp Asymptotic Minimaxity over Heterogeneous Signal-Strength Classes}
\label{subsec:sharp_minimax_heterogeneous}

We now turn to the second principal contribution of the paper by extending the normalized Hamming risk analysis developed in the previous subsection to the heterogeneous signal-strength framework introduced in Section~\ref{subsec:heterogeneous_framework}. Unlike the classical beta-min formulation, which characterizes the difficulty of sparse multiple testing through a common signal-strength threshold, the heterogeneous framework permits the nonzero means to vary arbitrarily across coordinates while preserving the same sparse asymptotic regime. As discussed in Section~\ref{subsec:heterogeneous_framework}, the natural heterogeneous benchmark is governed by the aggregate quantity $\Lambda_n(a)$ rather than the universal constant $\bar{\Phi}(b)$.

Throughout this subsection, we assume the heterogeneous signal-strength framework introduced in Section~\ref{subsec:heterogeneous_framework}, together with the asymptotic conditions \eqref{eq:general_sparse_regime} and the nominal-level condition under Assumption~\ref{ass:alpha}. The following theorem combines a uniform normalized Hamming risk upper bound for GBS with the heterogeneous FNR lower bound for sparsity-preserving procedures in \citet[Theorem~6]{ACR2024}.

\begin{theorem}[Sharp asymptotic minimaxity under normalized Hamming loss:
	heterogeneous signal-strength classes]
	\label{thm:hamming_general}
	Assume the sparse Gaussian sequence model \eqref{eq:model}, the heterogeneous signal-strength framework introduced in Section~\ref{subsec:heterogeneous_framework}, the sparsity condition \eqref{eq:general_sparse_regime}, and the nominal-level condition under Assumption~\ref{ass:alpha}. Fix a heterogeneous signal-strength vector
	\[
	a=(a_1,\ldots,a_{s_n}),
	\]
	and consider the corresponding parameter class $\Theta(a,s_n)$ defined in \eqref{eq:theta_general_a}. Then the GBS procedure satisfies
	\begin{equation}
		\sup_{\theta\in\Theta(a,s_n)}
		\frac{1}{s_n}
		\mathbb E_\theta
		\{V(\theta,\varphi^{\GBS})
		+
		T(\theta,\varphi^{\GBS})\}
		\le
		\Lambda_n(a)+o(1).
		\label{eq:hamming_general_upper}
	\end{equation}

Consequently, using the heterogeneous FNR lower bound for sparsity-preserving procedures from \citet[Theorem~6]{ACR2024},
\begin{equation}
\inf_{\varphi}\sup_{\theta\in\Theta(a,s_n)}
\frac{1}{s_n}\mathbb E_\theta\{V(\theta,\varphi)+T(\theta,\varphi)\}
=\sup_{\theta\in\Theta(a,s_n)}\frac{1}{s_n}
\mathbb E_\theta\{V(\theta,\varphi^{\GBS})+T(\theta,\varphi^{\GBS})\}
+o(1)=\Lambda_n(a)+o(1).
\label{eq:hamming_general_sharp}
\end{equation}
Thus GBS is sharply asymptotically minimax under normalized Hamming loss
for the heterogeneous class.
\end{theorem}

Theorem~\ref{thm:hamming_general} extends the sharp asymptotic minimaxity result of Theorem~\ref{thm:hamming-main} to the heterogeneous signal-strength framework. Indeed, when all nonzero means satisfy
\[
a_j=a_n^*+b,\qquad 1\le j\le s_n,
\]
the heterogeneous benchmark reduces to
\[
\Lambda_n(a)=\bar{\Phi}(b),
\]
so that the upper bound of Theorem~\ref{thm:hamming-main} is recovered as
a special case. More generally, the theorem extends the normalized
Hamming risk analysis of the GBS procedure beyond the classical
beta-min framework to the heterogeneous signal-strength classes
introduced by \citet{ACR2024}. In contrast to the beta-min setting,
where the asymptotic benchmark is determined by the universal constant
$\bar{\Phi}(b)$, the heterogeneous framework is governed by the aggregate
quantity $\Lambda_n(a)$, which reflects the collective contribution of
the individual signal strengths. Theorem~\ref{thm:hamming_general}
shows that the normalized Hamming risk of the GBS procedure remains
uniformly bounded above by this heterogeneous benchmark over the entire
parameter class $\Theta(a,s_n)$; the exact-support lower bound derived from \citet[Theorem~6]{ACR2024} makes the result sharp.

\begin{proof}[{\bf Proof of Theorem~\ref{thm:hamming_general}}]
	The false-positive contribution is controlled by the shifted order-statistic
	inequality for the GBS critical constants. The same argument used in
	Lemma~\ref{lem:false-positive} applies uniformly over all configurations
	with exactly $s_n$ nonzero coordinates. Hence, uniformly over
	$\theta\in\Theta(a,s_n)$,
	\[
	\frac{1}{s_n}
	\mathbb E_\theta V(\theta,\varphi^{\GBS})
	\le
	\frac{\alpha_n}{1-\alpha_n}
	\left(1+\frac{1}{s_n}\right)
	=
	o(1).
	\]
	By Lemma~\ref{lem:false_negative_general},
	\[
	\sup_{\theta\in\Theta(a,s_n)}
	\frac{1}{s_n}
	\mathbb E_\theta T(\theta,\varphi^{\GBS})
	\le
	\Lambda_n(a)+o(1).
	\]
	Combining the preceding two displays gives
	\[
	\sup_{\theta\in\Theta(a,s_n)}
	\frac{1}{s_n}
	\mathbb E_\theta
	\{V(\theta,\varphi^{\GBS})+T(\theta,\varphi^{\GBS})\}
	\le
	\Lambda_n(a)+o(1).
	\]
	This establishes the asserted asymptotic upper bound. For the matching lower bound, set
$M_n(a)=\inf_\varphi\sup_{\theta\in\Theta(a,s_n)}
\mathbb E_\theta(V+T)/s_n$. Since the reject-none rule has risk one,
choose $\varphi_n$ with maximal normalized Hamming risk at most
$M_n(a)+n^{-1}\le1+n^{-1}$. Exactly as in the proof of
Theorem~\ref{thm:hamming-main}, the inequality $R_n\le s_n+V$
and Markov's inequality imply, for every $B_n>1$,
\[
\sup_{\theta\in\Theta(a,s_n)}
\mathbb P_\theta(R_n>B_ns_n)
\le\frac{1+n^{-1}}{B_n-1}.
\]
Set $u_n=\log(n/s_n)$ and choose
\[
B_n=\min\{\log u_n+2,\ \exp(u_n^{1/4}/2)\}.
\]
Then $B_n\to\infty$ and $B_n^2\le\exp(u_n^{1/4})$.
We explicitly verify the tail restriction in
\citet[Theorem~6]{ACR2024}. In their Gaussian location model,
Assumption~1A holds with $a_n^*=\sqrt{2u_n}$ and, for example,
$\delta_n=u_n^{-1/4}$. Indeed, writing $x_n=\sqrt{2u_n}$,
Mills' asymptotic formula gives
\[
\log\left\{\frac{n}{s_n}\overline\Phi(x_n-\delta_n)\right\}
=u_n-\frac{(x_n-\delta_n)^2}{2}
-\log(x_n-\delta_n)-\frac12\log(2\pi)+o(1)
=\sqrt{2}\,u_n^{1/4}+o(u_n^{1/4}).
\]
Since $\sqrt{2}>1$, for all sufficiently large $n$,
\[
B_n^2\le\exp(u_n^{1/4})
\le\frac{n}{3s_n}\overline\Phi(a_n^*-\delta_n).
\]
Furthermore, $(n/s_n)\overline\Phi(a_n^*)\to0$ by Mills'
formula, and the Gaussian location family satisfies the
remaining monotone-likelihood-ratio, symmetry, and Lipschitz
conditions of their Assumption~1A. Hence the hypotheses of
\citet[Theorem~6]{ACR2024} are satisfied. The Markov bound above
also proves that $\varphi_n\in\mathcal S_B(\Theta(a,s_n))$
in the sense of their Definition~1.
The exact-support FNR minimax lower bound in
\citet[Theorem~6]{ACR2024} therefore gives
\[
M_n(a)+n^{-1}\ge
\sup_{\theta\in\Theta(a,s_n)}
\mathrm{FNR}(\theta,\varphi_n)
\ge\Lambda_n(a)-o(1).
\]
Combining this with the GBS upper bound proves
\eqref{eq:hamming_general_sharp}.
\end{proof}

We next consider the combined $\mathrm{FDP}+\mathrm{FNP}$ loss under the heterogeneous signal-strength framework. As in the beta-min setting, this criterion simultaneously measures the proportions of false discoveries and missed discoveries and therefore provides a complementary decision-theoretic assessment of multiple testing performance. The following theorem shows that the sharp asymptotic minimaxity of the GBS procedure continues to hold under this loss over the general parameter classes $\Theta(a,s_n)$.

\begin{theorem}[Sharp asymptotic minimaxity under the combined $\mathrm{FDP}+\mathrm{FNP}$ loss:
	heterogeneous signal-strength classes]
	\label{thm:fdp_fnp_general}
	Assume the sparse Gaussian sequence model \eqref{eq:model}, the heterogeneous signal-strength framework introduced in Section~\ref{subsec:heterogeneous_framework}, the sparsity condition \eqref{eq:general_sparse_regime}, and the nominal-level condition under Assumption~\ref{ass:alpha}. Fix a heterogeneous signal-strength vector
	\[
	a=(a_1,\ldots,a_{s_n}),
	\]
	and consider the corresponding parameter class $\Theta(a,s_n)$ defined in \eqref{eq:theta_general_a}. Then the GBS procedure satisfies
	\begin{equation}
		\sup_{\theta\in\Theta(a,s_n)}
		\mathbb E_\theta
		\left\{
		\mathrm{FDP}(\theta,\varphi^{\GBS})
		+
		\mathrm{FNP}(\theta,\varphi^{\GBS})
		\right\}
		\le
		\Lambda_n(a)+o(1).
		\label{eq:fdp_fnp_general_upper}
	\end{equation}
	Consequently, by the sharp minimax lower bound for the combined $\mathrm{FDR}+\mathrm{FNR}$ risk in \citet[Theorem~4, under Assumption~1A]{ACR2024}, which applies to the Gaussian location model and the exact-support heterogeneous class $\Theta(a,s_n)$, the GBS procedure is sharply asymptotically minimax under the combined $\mathrm{FDP}+\mathrm{FNP}$ loss over $\Theta(a,s_n)$.
\end{theorem}

The published Theorem~4 of \citet{ACR2024} gives the additive benchmark $\Lambda_n(a)+o(1)$ without requiring convergence of $\Lambda_n(a)$; in particular no limiting heterogeneous composition is assumed here.\par
Theorem~\ref{thm:fdp_fnp_general} complements Theorem~\ref{thm:hamming_general} by extending the sharp asymptotic minimaxity of the GBS procedure from normalized Hamming loss to the combined $\mathrm{FDP}+\mathrm{FNP}$ loss over the heterogeneous signal-strength classes. Consequently, the GBS procedure attains the exact sharp asymptotic minimax benchmark established by \citet{ACR2024} under both principal sparse multiple testing criteria, even in the substantially more general heterogeneous signal-strength setting.

\begin{proof}[{\bf Proof of Theorem~\ref{thm:fdp_fnp_general}}]
	Under independence, the GBS procedure controls the false discovery rate at
	level $\alpha_n$. Therefore,
	\[
	\sup_{\theta\in\Theta(a,s_n)}
	\mathbb E_\theta
	\{\mathrm{FDP}(\theta,\varphi^{\GBS})\}
	\le
	\alpha_n
	=
	o(1).
	\]
	Moreover, since $|S_\theta|=s_n$ for every $\theta\in\Theta(a,s_n)$,
	\[
	\mathrm{FNP}(\theta,\varphi^{\GBS})
	=
	\frac{T(\theta,\varphi^{\GBS})}{s_n}.
	\]
	Hence Lemma~\ref{lem:false_negative_general} gives
	\[
	\sup_{\theta\in\Theta(a,s_n)}
	\mathbb E_\theta
	\{\mathrm{FNP}(\theta,\varphi^{\GBS})\}
	\le
	\Lambda_n(a)+o(1).
	\]
	Combining the preceding two displays proves
	\eqref{eq:fdp_fnp_general_upper}.
	For the matching lower bound, the published \citet[Theorem~4, Assumption~1A]{ACR2024} establishes, without requiring $\Lambda_n(a)$ to converge,
\[
\inf_\varphi\sup_{\theta\in\Theta(a,s_n)}
\left\{\mathrm{FDR}(\theta,\varphi)+\mathrm{FNR}(\theta,\varphi)\right\}
=\Lambda_n(a)+o(1)
\]
for the Gaussian location model. Their parameter set agrees with
\eqref{eq:theta_general_a}, and their $\mathrm{FNR}$ equals our
expected $\mathrm{FNP}$ because the support size is exactly $s_n$.
The Gaussian Assumption~1A is verified in the proof of
Theorem~\ref{thm:hamming_general}; thus this lower bound and
\eqref{eq:fdp_fnp_general_upper} yield the asserted sharp equality.
\end{proof}

\section{Discussion}
\label{sec:discussion}

The central statistical lesson is that stringent sequential stopping need not entail an additional first-order loss in sparse signal recovery. FDR control alone can be achieved by a procedure that rarely rejects, so it says little about how many genuine signals are found. For GBS, the requirement that \emph{every} preceding ordered $p$-value satisfy its critical inequality might have suggested a loss of power relative to a step-up rule. At the critical Gaussian detection boundary, however, the resulting missed-signal risk reaches the same sharp minimax benchmark as procedures with quite different rejection architectures. The significance is an attainability statement for an established boundary, rather than a new detection boundary or a universal claim about step-down rules.

The heterogeneous formulation offers a second interpretation. When signals differ in strength, recovery difficulty is not adequately described by the weakest signal or by a single representative amplitude. Instead, the benchmark $\Lambda_n(a)$ averages the individual Gaussian missed-detection probabilities relative to the sparse detection boundary. Thus a population containing both readily detectable and borderline signals can have a nontrivial aggregate recovery limit even when no single beta-min offset describes it. The GBS result shows that its sequential stopping constraint does not increase this aggregate benchmark to first order. This conclusion concerns expected normalized errors; it does not guarantee exact support recovery or uniform detection of every weak coordinate.

An important question is how much of this phenomenon belongs to GBS and how much reflects more general features of step-down testing. A possible theory would identify verifiable restrictions on a critical-value sequence that jointly prevent excessive null rejections and premature stopping along the entire initial segment of ordered $p$-values. Conditions involving only the critical values near a final rejection rank may be inadequate, since one unusually stringent early rank can prevent all subsequent rejections. Finding a useful set of sufficient conditions, and examples marking their limits, remains an open problem. Such a development would require additional mathematics beyond the present procedure-specific crossing argument.

The independence assumption is another substantive boundary of the present conclusions. With correlated observations, the evidence supplied by one coordinate can affect the interpretation of other coordinates; for residual-based sequential rules, the statistics themselves may change after each rejection. In that setting, neither the independent-null order-statistic control nor the current signal-crossing argument transfers automatically. A dependence-aware sharp minimax theory would have to account for the evolving residual process and establish the appropriate benchmark under the chosen covariance model. This offers a direction in which the interplay between dependence, sequential evidence, and sparse recovery may differ qualitatively from the independent case.

Finally, the sufficient calibration condition on the vanishing nominal level leaves open a more refined statistical question: how rapidly may a step-down procedure reduce its tolerance for false discoveries before premature stopping creates a detectable loss of power? Establishing a sharp boundary for this trade-off, rather than merely sufficient conditions, could clarify when the apparent conservatism of step-down testing becomes statistically consequential.

\appendix

\section*{Appendix: Technical Lemmas and Proofs}


The proofs rest on two distinct controls. First, the shifted uniform order-statistic inequality established by \citet{Ghosh2026GBSABOS} bounds the expected number of false positives; in the present sparsity regime this contribution is $o(s_n)$. Second, the deterministic and empirical signal-crossing arguments developed here bound false negatives by showing that the GBS path passes \emph{all} ordered $p$-value constraints up to a sufficiently large rejection rank. A crossing at the final rank alone would not suffice for a step-down procedure. The latter arguments require special care at the earliest ranks and, for Theorem~\ref{thm:at-most-hamming}, uniformly over the actual support size. The heterogeneous extension replaces the single beta-min offset with the aggregate missed-signal benchmark $\Lambda_n(a)$. The minimax lower bounds are invoked from \citet{ACR2024}; the new work establishes the corresponding GBS upper bounds. We collect the supporting lemmas and their proofs below.

We begin with the false-positive and signal-crossing estimates needed for Theorem~\ref{thm:hamming-main}. The shifted order-statistic inequality is recalled from \citet{Ghosh2026GBSABOS} without reproducing its original proof. After completing those arguments, we establish the support-size-uniform extension used in Theorem~\ref{thm:at-most-hamming}, before turning to heterogeneous signal strengths.

\begin{lemma}[Shifted order-statistic inequality]\label{lem:shifted-os}
	Let $U_{(1)}\le\cdots\le U_{(m)}$ be the order statistics of $m$ independent $U(0,1)$ random variables.  Let $k\ge0$ and $0<\alpha<1$.  For $\ell=1,\ldots,m$, define
	\[
	b_{\ell,k}^{(m)}(\alpha)
	=\frac{(k+\ell)\alpha}{m+k+1-(k+\ell)(1-\alpha)}.
	\]
	Then
	\[
	\sum_{j=1}^m
	\mathbb P\left(U_{(\ell)}\le b_{\ell,k}^{(m)}(\alpha),\ \ell=1,\ldots,j\right)
	\le \frac{\alpha}{1-\alpha}(k+1).
	\]
\end{lemma}
\begin{proof}
	See \citet{Ghosh2026GBSABOS}.
\end{proof}

Lemma~\ref{lem:shifted-os} provides the principal ingredient for controlling the false-positive
component of the normalized Hamming risk. Unlike the order-statistic bounds
typically employed in the analysis of BH-type procedures, the present inequality
is tailored specifically to the shifted critical sequence of the GBS step-down
procedure. Consequently, it provides a direct characterization of the null
contribution without relying on threshold-localization arguments. Since the
proof is identical to that in \citet{Ghosh2026GBSABOS}, it is omitted here.

\begin{lemma}[False-positive bound for the GBS procedure]
	\label{lem:false-positive}
	Under Assumptions~\ref{ass:sparse}--\ref{ass:alpha}, for every parameter vector
	$\theta\in\mathbb R^n$ satisfying $|S_\theta|=s_n$,
	\[
	\mathbb E_{\theta} V(\theta,\varphi^{\GBS})
	\le
	\frac{\alpha_n}{1-\alpha_n}(s_n+1).
	\]
	Consequently,
	\[
	\sup_{\theta:\, |S_\theta|=s_n}
	\frac{1}{s_n}
	\mathbb E_{\theta}V(\theta,\varphi^{\GBS})
	=o(1).
	\]
	In particular, the same conclusion holds uniformly over both
	$\Theta_b$ and $\Theta(a,s_n)$.
\end{lemma}

Lemma~\ref{lem:false-positive} establishes that the false-positive contribution to the normalized Hamming risk is asymptotically negligible, uniformly over the beta-min class $\Theta_b$. The key point is that the expected number of false rejections is bounded by a term of order $\alpha_n \sn$, which becomes negligible after normalization by $\sn$ because $\alpha_n\to0$. Consequently, the sharp asymptotic minimax behavior of the GBS procedure is governed by its false-negative contribution. This reduction is essential: once the null contribution is controlled, the remaining task is to show that the GBS rejection path captures the correct asymptotic fraction of non-null coordinates. This is achieved through the deterministic and empirical signal-crossing arguments developed below.

\begin{proof}[{\bf Proof of Lemma~\ref{lem:false-positive}}]
	Fix $\theta\in\mathbb R^n$ such that $|S_\theta|=s_n$, and write
	$S=S_\theta$. Conditional on the signal $p$-values, there are
	$n-s_n$ null $p$-values, which are independent $U(0,1)$ random
	variables and are independent of the signal $p$-values. Let
	\[
	P^0_{(1)}\le \cdots \le P^0_{(n-s_n)}
	\]
	denote the ordered null $p$-values.
	
	Fix any $j=1,\ldots,n-s_n$, and suppose that the GBS procedure makes
	at least $j$ false rejections. Then the $j$ smallest null $p$-values
	must be rejected. For $\ell=1,\ldots,j$, let $R_\ell$ denote the
	overall rank of $P^0_{(\ell)}$ among all $n$ ordered $p$-values.
	Since there are exactly $s_n$ signal $p$-values, at most $s_n$ signal
	$p$-values can be smaller than $P^0_{(\ell)}$. Moreover, among the
	null $p$-values, exactly $\ell-1$ values are smaller than
	$P^0_{(\ell)}$. Hence $R_\ell\le s_n+\ell$. If the null hypothesis
	corresponding to $P^0_{(\ell)}$ is rejected, then, by the definition
	of the GBS step-down procedure,
	$P^0_{(\ell)}\le c_{R_\ell}$. Since the GBS critical constants are
	nondecreasing in the rank, it follows that
	$P^0_{(\ell)}\le c_{R_\ell}\le c_{s_n+\ell}$ for every
	$\ell=1,\ldots,j$. Therefore, conditional on the signal $p$-values,
	\begin{equation}
		\{V(\theta,\varphi^{\GBS})\ge j\}
		\subseteq
		\{P^0_{(\ell)}\le c_{s_n+\ell},\ \ell=1,\ldots,j\}.
		\label{eq:false-positive-event}
	\end{equation}
	Taking conditional probabilities in
	\eqref{eq:false-positive-event} gives
	\begin{equation}
		\mathbb P_\theta\!\left(
		V(\theta,\varphi^{\GBS})\ge j
		\,\middle|\,
		\{P_i:i\in S_\theta\}
		\right)
		\le
		\mathbb P_\theta\!\left(
		P^0_{(\ell)}\le c_{s_n+\ell},
		\ \ell=1,\ldots,j
		\,\middle|\,
		\{P_i:i\in S_\theta\}
		\right).
		\label{eq:false-positive-conditional}
	\end{equation}
	Since, conditional on the signal $p$-values, the ordered null
	$p$-values have the same joint distribution as the order statistics
	of $n-s_n$ independent $U(0,1)$ random variables, the right-hand
	side of \eqref{eq:false-positive-conditional} is equal to
	\begin{equation}
		\mathbb P\!\left(
		U_{(\ell)}\le c_{s_n+\ell},
		\ \ell=1,\ldots,j
		\right),
		\label{eq:false-positive-uniform}
	\end{equation}
	where $U_{(1)}\le\cdots\le U_{(n-s_n)}$ denote these uniform order
	statistics. Taking expectations with respect to the signal $p$-values in
	\eqref{eq:false-positive-conditional}, together with 	\eqref{eq:false-positive-uniform}, yields
	\begin{equation}
		\mathbb P_\theta\!\left(
		V(\theta,\varphi^{\GBS})\ge j
		\right)
		\le
		\mathbb P\!\left(
		U_{(\ell)}\le c_{s_n+\ell},
		\ \ell=1,\ldots,j
		\right).
		\label{eq:false-positive-unconditional}
	\end{equation}
	
	Now,
	\begin{equation}
		c_{s_n+\ell}
		=
		\frac{(s_n+\ell)\alpha_n}
		{n+1-(s_n+\ell)(1-\alpha_n)}
		=
		\frac{(s_n+\ell)\alpha_n}
		{(n-s_n)+s_n+1-(s_n+\ell)(1-\alpha_n)}.
		\label{eq:gbs-shifted-critical}
	\end{equation}
	Using
	\eqref{eq:false-positive-unconditional},
	\eqref{eq:gbs-shifted-critical},
	and Lemma~\ref{lem:shifted-os}, applied with
	$m=n-s_n$ and $k=s_n$, we obtain
	\begin{equation}
		\begin{aligned}
			\mathbb E_\theta V(\theta,\varphi^{\GBS})
			&=
			\sum_{j=1}^{n-s_n}
			\mathbb P_\theta\!\left(
			V(\theta,\varphi^{\GBS})\ge j
			\right) \\
			&\le
			\frac{\alpha_n}{1-\alpha_n}(s_n+1).
		\end{aligned}
		\label{eq:false-positive-expectation}
	\end{equation}
	Dividing both sides of
	\eqref{eq:false-positive-expectation} by $s_n$ gives
	\begin{equation}
		\frac{1}{s_n}
		\mathbb E_\theta V(\theta,\varphi^{\GBS})
		\le
		\frac{\alpha_n}{1-\alpha_n}
		\left(1+\frac{1}{s_n}\right).
		\label{eq:false-positive-normalized-bound}
	\end{equation}
	Since $\alpha_n\to0$ and $s_n\to\infty$, the right-hand side of
	\eqref{eq:false-positive-normalized-bound} is $o(1)$, uniformly over
	all $\theta\in\mathbb R^n$ satisfying $|S_\theta|=s_n$. This proves
	the lemma.
\end{proof}

We next turn to the false-negative contribution. For $\theta\in\Theta_b$, let
$Q_{(1)}\le\cdots\le Q_{(\sn)}$ denote the ordered signal $p$-values, and define
\[
N_1(t)
=
\sum_{i\in S_\theta}\mathbf 1\{P_i\le t\}
\]
to be the number of signal $p$-values below $t$. The key objective is to show that, for every $\eps\in(0,\Phi(b))$,
\[
\sup_{\theta\in\Theta_b}
\mathbb P_\theta
\left\{
R_n^{\GBS}
<
(\Phi(b)-\eps)\sn
\right\}
\to0 \qquad \textrm{as } n\to\infty.
\]
Together with the asymptotic negligibility of the false-positive contribution,
this implies that, uniformly over $\Theta_b$, the GBS procedure misses at most a
$(\bar\Phi(b)+\eps)$ fraction of the signals asymptotically. The following elementary
monotonicity lemma will be used to propagate a deterministic signal-crossing bound
from one point of the GBS critical sequence to an entire range of ranks.

\begin{lemma}[Monotonicity of conditional tail averages]
	\label{lem:conditional_tail_average}
	Let $Y$ be a random variable with density $f$ supported on $[0,\infty)$, and let
	$g:[0,\infty)\to\mathbb R$ be nondecreasing. Suppose that the following integrals are finite and that
	\[
	B(z):=\int_z^\infty f(y)\,dy>0,\qquad z\ge0.
	\]
	Define
	\begin{equation}
		M(z)
		=
		\frac{\int_z^\infty g(y)f(y)\,dy}
		{\int_z^\infty f(y)\,dy},
		\qquad z\ge0.
		\label{eq:conditional_tail_average}
	\end{equation}
	Then $M(z)$ is nondecreasing in $z$.
\end{lemma}

Lemma~\ref{lem:conditional_tail_average} provides the principal analytical tool underlying the deterministic signal-crossing analysis developed below. In the present context, it will be applied to suitable conditional tail expectations associated with the Gaussian alternative distribution to establish the monotonicity of the signal distribution relative to the GBS critical sequence. This monotonicity property plays a crucial role in showing that, once the ordered signal $p$-values cross the GBS critical values at a suitable location, the crossing behavior propagates throughout the relevant range of ranks. Consequently, Lemma~\ref{lem:conditional_tail_average} forms the key analytical ingredient in the deterministic characterization of the GBS rejection path.

\begin{proof}[{\bf Proof of Lemma~\ref{lem:conditional_tail_average}}]
	Fix $0\le z_1<z_2$. Define
	\begin{equation}
		A_j
		=
		\int_{z_j}^{\infty} g(y)f(y)\,dy,
		\qquad
		B_j
		=
		\int_{z_j}^{\infty} f(y)\,dy,
		\qquad j=1,2.
		\label{eq:tail_average_Aj_Bj}
	\end{equation}
	Also define
	\begin{equation}
		C
		=
		\int_{z_1}^{z_2} f(y)\,dy,
		\qquad
		D
		=
		\int_{z_1}^{z_2} g(y)f(y)\,dy.
		\label{eq:tail_average_C_D}
	\end{equation}
	By construction,
	\begin{equation}
		B_1=B_2+C,
		\qquad
		A_1=A_2+D.
		\label{eq:A_B_decomposition}
	\end{equation}
	Since $B(z)>0$ for every $z\ge0$, we have
	\begin{equation}
		B_1>0,
		\qquad
		B_2>0.
		\label{eq:B_positive}
	\end{equation}
	Therefore, to prove that $M(z_1)\le M(z_2)$, it is enough to show
	\begin{equation}
		\frac{A_1}{B_1}
		\le
		\frac{A_2}{B_2}.
		\label{eq:tail_average_goal_ratio}
	\end{equation}
	Using \eqref{eq:A_B_decomposition}, the desired inequality
	\eqref{eq:tail_average_goal_ratio} is equivalent to
	\begin{equation}
		\frac{A_2+D}{B_2+C}
		\le
		\frac{A_2}{B_2}.
		\label{eq:tail_average_equiv_ratio}
	\end{equation}
	By \eqref{eq:B_positive}, \eqref{eq:tail_average_equiv_ratio} is equivalent to
	\begin{align}
		(A_2+D)B_2
		&\le
		A_2(B_2+C)
		\nonumber\\
		\Longleftrightarrow\qquad
		DB_2
		&\le
		C A_2.
		\label{eq:tail_average_DB_AC}
	\end{align}
	Thus it remains to prove \eqref{eq:tail_average_DB_AC}.
	
	Since $g$ is nondecreasing, for every $y\in[z_1,z_2]$ and every
	$u\in[z_2,\infty)$,
	\begin{equation}
		g(y)\le g(u).
		\label{eq:g_monotone_y_u}
	\end{equation}
	Because $f$ is a density, $f(y)f(u)\ge0$. Multiplying both sides of
	\eqref{eq:g_monotone_y_u} by $f(y)f(u)$ gives
	\begin{equation}
		g(y)f(y)f(u)
		\le
		g(u)f(y)f(u).
		\label{eq:pointwise_product_ineq}
	\end{equation}
	Integrating both sides of \eqref{eq:pointwise_product_ineq} over
	$(y,u)\in[z_1,z_2]\times[z_2,\infty)$ yields
	\begin{align}
		\int_{z_1}^{z_2}\int_{z_2}^{\infty}
		g(y)f(y)f(u)\,du\,dy
		&\le
		\int_{z_1}^{z_2}\int_{z_2}^{\infty}
		g(u)f(y)f(u)\,du\,dy.
		\label{eq:double_integral_ineq}
	\end{align}
	By Fubini's theorem (the relevant integrals are absolutely finite), the left-hand side of
	\eqref{eq:double_integral_ineq} equals
	\begin{align}
		\int_{z_1}^{z_2}\int_{z_2}^{\infty}
		g(y)f(y)f(u)\,du\,dy
		&=
		\left\{
		\int_{z_1}^{z_2} g(y)f(y)\,dy
		\right\}
		\left\{
		\int_{z_2}^{\infty} f(u)\,du
		\right\}
		\nonumber\\
		&=
		DB_2.
		\label{eq:left_double_DB}
	\end{align}
	Similarly, the right-hand side of \eqref{eq:double_integral_ineq} equals
	\begin{align}
		\int_{z_1}^{z_2}\int_{z_2}^{\infty}
		g(u)f(y)f(u)\,du\,dy
		&=
		\left\{
		\int_{z_1}^{z_2} f(y)\,dy
		\right\}
		\left\{
		\int_{z_2}^{\infty} g(u)f(u)\,du
		\right\}
		\nonumber\\
		&=
		CA_2.
		\label{eq:right_double_CA}
	\end{align}
	Combining \eqref{eq:double_integral_ineq},
	\eqref{eq:left_double_DB}, and \eqref{eq:right_double_CA}, we obtain
	\begin{equation}
		DB_2\le CA_2.
		\label{eq:DB_leq_CA_final}
	\end{equation}
	Hence \eqref{eq:tail_average_DB_AC} holds, and therefore
	\eqref{eq:tail_average_goal_ratio} follows. Thus
	\begin{equation}
		M(z_1)\le M(z_2).
		\label{eq:M_z1_leq_M_z2}
	\end{equation}
	Since $0\le z_1<z_2$ were arbitrary, $M$ is nondecreasing on $[0,\infty)$.
\end{proof}

\begin{lemma}[Initial crossing]
	\label{lem:initial-crossing}
	Assume that
	\begin{equation}
		s_n\to\infty,
		\qquad
		\frac{n}{s_n}\to\infty, \qquad \textrm{as } n\to\infty,
		\label{eq:initial_sparse_regime}
	\end{equation}
	and that
	\begin{equation}
		\alpha_n\to0,
		\qquad
		\log(1/\alpha_n)=o(a_n),
		\qquad
		a_n=\sqrt{2\log(n/s_n)}.
		\label{eq:initial_alpha_condition}
	\end{equation}
	Then
	\begin{equation}
		\inf_{\theta\in\Theta_b}
		\mathbb E_{\theta}N_1(c_1)
		\to\infty \qquad \textrm{as } n\to\infty,
		\label{eq:initial_crossing_conclusion}
	\end{equation}
	where
	\begin{equation}
		N_1(t)
		=
		\sum_{i\in S_\theta}1\{P_i\le t\},
		\qquad
		c_1=\frac{\alpha_n}{n+\alpha_n}.
		\label{eq:N1_c1_initial}
	\end{equation}
\end{lemma}

Lemma~\ref{lem:initial-crossing} establishes that, asymptotically, the expected number of signal $p$-values below the first GBS critical value diverges uniformly over the beta-min class $\Theta_b$. Thus, despite the extreme conservativeness of the initial critical value, sufficiently many signals are expected to lie below it. This result provides the first indication that the GBS critical sequence and the ordered signal $p$-values necessarily intersect in the sparse asymptotic regime. It therefore serves as the starting point of the deterministic signal-crossing analysis developed in the subsequent results.

\begin{proof}[{\bf Proof of Lemma~\ref{lem:initial-crossing}}]
	Fix $\theta\in\Theta_b$. Under the sparsity assumption, $a_n\to\infty$ as $n\to\infty$. Hence,
	$a_n+b>0$ for all sufficiently large $n$.
	For $0<t<1$, let
	\begin{equation}
		z_t=\Phi^{-1}(1-t/2).
		\label{eq:zt_initial_def}
	\end{equation}
If $X\sim N(\mu,1)$, define
\begin{equation}
	G_\mu(t)
	=
	\mathbb P_\mu\{2\bar\Phi(|X|)\le t\},
	\label{eq:Gmu_initial_def}
\end{equation}
that is, $G_\mu$ denotes the cumulative distribution function of the two-sided $p$-value for testing $H_0:\mu=0$ versus $H_1:\mu\neq0$. Then
	\begin{align}
		G_\mu(t)
		&=
		\mathbb P_\mu(|X|\ge z_t)
		\nonumber\\
		&=
		\bar\Phi(z_t-\mu)+\bar\Phi(z_t+\mu).
		\label{eq:Gmu_initial_representation}
	\end{align}
By symmetry,
\[
G_\mu(t)
=
\mathbb P_\mu(|X|\ge z_t)
=
\bar\Phi(z_t-\mu)
+
\bar\Phi(z_t+\mu),
\]
and hence $G_{-\mu}(t)=G_\mu(t)$. Therefore, it suffices to consider
$\mu\ge0$. For $\mu\ge0$,
\begin{align}
	\frac{\partial}{\partial\mu}G_\mu(t)
	&=
	\phi(z_t-\mu)-\phi(z_t+\mu)
	\ge0,
	\label{eq:Gmu_initial_monotone}
\end{align}
because $z_t\ge0$, $\mu\ge0$, and consequently
$|z_t-\mu|\le z_t+\mu$. Hence $G_\mu(t)$ is nondecreasing as a function of $|\mu|$.

Since $G_\mu(t)$ is symmetric in $\mu$ and nondecreasing as a function of
$|\mu|$, and since $|\theta_i|\ge a_n+b$ for every $i\in S_\theta$, we have
\[
G_{\theta_i}(t)
=
G_{|\theta_i|}(t)
\ge
G_{a_n+b}(t),
\qquad i\in S_\theta.
\]
Therefore,
\begin{align}
	\mathbb E_\theta N_1(c_1)
	&=
	\sum_{i\in S_\theta}
	\mathbb P_{\theta_i}(P_i\le c_1)
	=
	\sum_{i\in S_\theta}
	G_{\theta_i}(c_1)
	\ge
	s_nG_{a_n+b}(c_1).
	\label{eq:initial_expectation_lower_1}
\end{align} 
Using \eqref{eq:Gmu_initial_representation},
\begin{align}
	G_{a_n+b}(c_1)
	&=
	\bar\Phi(z_{c_1}-a_n-b)
	+
	\bar\Phi(z_{c_1}+a_n+b)
	\nonumber\\
	&\ge
	\bar\Phi(z_{c_1}-a_n-b).
	\label{eq:G_initial_one_tail}
\end{align}

Therefore,
\begin{equation}
	\mathbb E_{\theta}N_1(c_1)
	\ge
	s_n\bar\Phi(z_{c_1}-a_n-b).
	\label{eq:initial_expectation_lower_2}
\end{equation}
It remains to show that the right-hand side of
\eqref{eq:initial_expectation_lower_2} diverges. Put
\[
L_n=\log(n/s_n),\qquad
S_n=\log s_n,\qquad
A_n=\log(1/\alpha_n).
\]
Then $a_n=\sqrt{2L_n}$, $L_n\to\infty$, $S_n\to\infty$, and
$A_n=o(\sqrt{L_n})$. Since $c_1=\alpha_n/(n+\alpha_n)$ and
$z_{c_1}=\bar\Phi^{-1}(c_1/2)$, the Gaussian quantile expansion gives
\[
z_{c_1}
=
\sqrt{2(L_n+S_n+A_n)}+o(1).
\]
Let
\[
\Delta_n
=
\sqrt{2(L_n+S_n+A_n)}-\sqrt{2L_n}.
\]
Then
\[
z_{c_1}-a_n-b=\Delta_n-b+o(1).
\]
By Mills' lower bound,
\[
s_n\bar\Phi(z_{c_1}-a_n-b)
\ge
C\exp\left\{
S_n-\frac{1}{2}(\Delta_n-b+o(1))_+^2
-\log\bigl(1+(\Delta_n-b+o(1))_+\bigr)
\right\}
\]
for a universal constant $C>0$. Thus it is enough to show that
\[
S_n-\frac{\Delta_n^2}{2}
-C(1+\Delta_n)-\log(1+\Delta_n)\to\infty
\]
for any fixed constant $C>0$. Now
\[
\frac{\Delta_n^2}{2}
=
S_n+A_n-\sqrt{2L_n}\Delta_n,
\]
and therefore
\[
S_n-\frac{\Delta_n^2}{2}
=
\sqrt{2L_n}\Delta_n-A_n.
\]
Moreover,
\[
\sqrt{2L_n}\Delta_n
=
\frac{2\sqrt{L_n}(S_n+A_n)}
{\sqrt{L_n+S_n+A_n}+\sqrt{L_n}}.
\]
Since $S_n\to\infty$ and $A_n=o(\sqrt{L_n})$, the last display implies
\[
\sqrt{2L_n}\Delta_n-A_n
-
C(1+\Delta_n)-\log(1+\Delta_n)
\to\infty.
\]
Consequently,
\[
s_n\bar\Phi(z_{c_1}-a_n-b)\to\infty.
\]
Combining this with \eqref{eq:initial_expectation_lower_2} gives
\[
\inf_{\theta\in\Theta_b}
\mathbb E_\theta N_1(c_1)\to\infty.
\]
This concludes the proof of Lemma~\ref{lem:initial-crossing}.
\end{proof}

\begin{lemma}[Folded-normal tail-ratio monotonicity]
	\label{lem:folded_normal_tail_ratio}
	For every $\mu>0$, the function
	\begin{equation}
		z\mapsto
		\widetilde R_\mu(z)
		=
		\frac{\bar\Phi(z-\mu)+\bar\Phi(z+\mu)}
		{2\bar\Phi(z)}
		\label{eq:folded_ratio}
	\end{equation}
	is nondecreasing on $[0,\infty)$.
\end{lemma}

Lemma~\ref{lem:folded_normal_tail_ratio} specializes the general monotonicity principle of Lemma~\ref{lem:conditional_tail_average} to the folded-normal distribution arising under the Gaussian sequence model. The monotonicity of the tail-ratio function $\widetilde R_\mu$ implies that the relative enrichment of signal $p$-values over null $p$-values increases monotonically with the threshold. This structural property forms the key probabilistic ingredient in the deterministic signal-crossing analysis, as it guarantees that once the ordered signal $p$-values intersect the GBS critical sequence, the crossing propagates in a controlled manner over the relevant range of ranks. Consequently, Lemma~\ref{lem:folded_normal_tail_ratio} provides the Gaussian-specific monotonicity property underlying the subsequent deterministic crossing theorem.

\begin{proof}[{\bf Proof of Lemma~\ref{lem:folded_normal_tail_ratio}}]
	Let $Y=|Z|$, where $Z\sim N(0,1)$. Then $Y$ has density
	\begin{equation}
		f_0(y)=2\phi(y)\mathbf 1_{\{y\ge0\}}.
		\label{eq:folded_null_density}
	\end{equation}
	If $X\sim N(\mu,1)$, then $|X|$ has the folded normal density
	\begin{equation}
		f_\mu(y)=\{\phi(y-\mu)+\phi(y+\mu)\}\mathbf 1_{\{y\ge0\}}.
		\label{eq:folded_signal_density}
	\end{equation}
	
	For $y\ge0$, using \eqref{eq:folded_null_density} and
	\eqref{eq:folded_signal_density}, we obtain
	\begin{align}
		\frac{f_\mu(y)}{f_0(y)}
		&=
		\frac{\phi(y-\mu)+\phi(y+\mu)}
		{2\phi(y)}
		\nonumber\\
		&=
		\frac{
			\exp\{\mu y-\mu^2/2\}
			+
			\exp\{-\mu y-\mu^2/2\}
		}{2}
		\nonumber\\
		&=
		\exp\{-\mu^2/2\}\cosh(\mu y).
		\label{eq:folded_lr}
	\end{align}
	Since $\mu>0$,
	\begin{equation}
		\frac{d}{dy}
		\left[
		\exp\{-\mu^2/2\}\cosh(\mu y)
		\right]
		=
		\mu\exp\{-\mu^2/2\}\sinh(\mu y)
		\ge0,
		\qquad y\ge0.
		\label{eq:folded_lr_monotone}
	\end{equation}
	Thus
	\begin{equation}
		g_\mu(y):=\frac{f_\mu(y)}{f_0(y)}
		\label{eq:gmu_definition}
	\end{equation}
	is nondecreasing on $[0,\infty)$.
	
	Next, for every $z\ge0$,
	\begin{align}
		\int_z^\infty f_\mu(y)\,dy
		&=
		\bar\Phi(z-\mu)+\bar\Phi(z+\mu),
		\label{eq:folded_signal_tail}
	\end{align}
	whereas
	\begin{equation}
		\int_z^\infty f_0(y)\,dy
		=
		2\bar\Phi(z).
		\label{eq:folded_null_tail}
	\end{equation}
	Therefore,
	\begin{equation}
		\widetilde R_\mu(z)
		=
		\frac{\int_z^\infty f_\mu(y)\,dy}
		{\int_z^\infty f_0(y)\,dy}.
		\label{eq:folded_ratio_tail_integrals}
	\end{equation}
	Since $f_\mu=g_\mu f_0$ on $[0,\infty)$, we may rewrite
	\eqref{eq:folded_ratio_tail_integrals} as
	\begin{align}
		\widetilde R_\mu(z)
		&=
		\frac{
			\int_z^\infty g_\mu(y) f_0(y)\,dy
		}
		{
			\int_z^\infty f_0(y)\,dy
		}.
		\label{eq:folded_ratio_conditional_average}
	\end{align}
	The denominator is strictly positive for every $z\ge0$, and the
	numerator is finite since $g_\mu f_0=f_\mu$ is a density on
	$[0,\infty)$. Lemma~\ref{lem:conditional_tail_average}, applied with
	$f=f_0$ and $g=g_\mu$, now implies that the right-hand side of
	\eqref{eq:folded_ratio_conditional_average} is nondecreasing in $z$.
	Thus $\widetilde R_\mu(z)$ is nondecreasing on $[0,\infty)$.
\end{proof}

Having established the analytical ingredients required for the signal-crossing analysis, we now prove the central deterministic crossing result that identifies the asymptotic location at which the signal process overtakes the GBS critical sequence.

\begin{lemma}[Deterministic signal crossing]
	\label{lem:deterministic_signal_crossing}
	Assume that
	\begin{equation}
		s_n\to\infty,
		\qquad
		\frac{n}{s_n}\to\infty,
		\label{eq:sparse_regime_crossing}
	\end{equation}
	and that the GBS nominal level satisfies
	\begin{equation}
		\alpha_n\to0,
		\qquad
		\log(1/\alpha_n)=o(a_n),
		\qquad
		a_n=\sqrt{2\log(n/s_n)}.
		\label{eq:alpha_condition_crossing}
	\end{equation}
	Fix $b\in\mathbb R$, and define
	\begin{equation}
		q_b=\Phi(b).
		\label{eq:qb_definition}
	\end{equation}
	For every $\varepsilon\in(0,q_b)$, there exist
	$\eta_\varepsilon>0$ and $n_\varepsilon\ge1$ such that, for all
	$n\ge n_\varepsilon$ and every $\theta\in\Theta_b$,
	\begin{equation}
		G_{\theta_i}\!\left(
		q\alpha_n\frac{s_n}{n}
		\right)
		\ge
		(1+\eta_\varepsilon)q,
		\qquad
		i\in S_\theta,
		\label{eq:deterministic_crossing_conclusion}
	\end{equation}
	uniformly over
	\begin{equation}
		0<q\le q_b-\varepsilon.
		\label{eq:q_range_crossing}
	\end{equation}
\end{lemma}

Lemma~\ref{lem:deterministic_signal_crossing} provides the fundamental deterministic characterization of the GBS rejection path. It shows that, uniformly over the beta-min class $\Theta_b$, the expected proportion of signal $p$-values below the GBS critical sequence exceeds the corresponding proportion of critical values throughout the entire range $0<q\le\Phi(b)-\varepsilon$. In other words, the deterministic signal process lies strictly above the GBS critical sequence up to the asymptotically optimal crossing point determined by the sparse detection boundary. This result identifies the deterministic location of the signal crossing and establishes the benchmark against which the stochastic behavior of the GBS rejection path will subsequently be compared. The next result shows that the empirical signal process closely tracks this deterministic approximation with probability tending to one, thereby transferring the deterministic crossing phenomenon to the actual GBS rejection path.

\begin{proof}[{\bf Proof of Lemma~\ref{lem:deterministic_signal_crossing}}]
	Recall from the proof of Lemma~\ref{lem:initial-crossing} that
	\[
	G_\mu(t)
	=
	\mathbb P_\mu\{2\bar\Phi(|X|)\le t\},
	\qquad X\sim N(\mu,1),
	\]
	is the distribution function of the two-sided $p$-value for testing
	$H_0:\mu=0$ versus $H_1:\mu\neq0$. Moreover,
	\begin{equation}
		G_\mu(t)
		=
		\bar\Phi(z_t-\mu)+\bar\Phi(z_t+\mu),
		\qquad
		z_t=\Phi^{-1}(1-t/2),
		\label{eq:Gmu_representation}
	\end{equation}
	and $G_\mu(t)$ is symmetric in $\mu$ and nondecreasing as a function
	of $|\mu|$.
	
	Fix $\varepsilon\in(0,q_b)$, and define
	\begin{equation}
		q_\varepsilon=q_b-\varepsilon,
		\qquad
		t_\varepsilon=q_\varepsilon\alpha_n\frac{s_n}{n}.
		\label{eq:qeps_teps_def}
	\end{equation}
	Since $q_\varepsilon$ is fixed and positive,
	\begin{align}
		\log(1/t_\varepsilon)
		&=
		\log(n/s_n)+\log(1/\alpha_n)+\log(1/q_\varepsilon)
		\nonumber\\
		&=
		\frac{a_n^2}{2}+o(a_n),
		\label{eq:log_teps_an}
	\end{align}
	where we used $a_n^2=2\log(n/s_n)$ and
	$\log(1/\alpha_n)=o(a_n)$.
	
	By the Gaussian quantile expansion,
	\[
	z_t^2
	=
	2\log(1/t)
	+
	O\{\log\log(1/t)\},
	\qquad t\downarrow0.
	\]
	Hence, by \eqref{eq:log_teps_an},
	\begin{align}
		z_{t_\varepsilon}^2
		&=
		2\log(1/t_\varepsilon)
		+
		O\{\log\log(1/t_\varepsilon)\}
		\nonumber\\
		&=
		a_n^2
		+
		o(a_n)
		+
		O\{\log\log(1/t_\varepsilon)\}.
		\label{eq:zteps_squared_intermediate}
	\end{align}
	Since
	\[
	\log(1/t_\varepsilon)
	=
	\frac{a_n^2}{2}
	+
	o(a_n)
	\asymp
	a_n^2,
	\]
	we have
	\[
	\log\log(1/t_\varepsilon)
	=
	O(\log a_n)
	=
	o(a_n).
	\]
	Therefore, \eqref{eq:zteps_squared_intermediate} yields
	\begin{equation}
		z_{t_\varepsilon}^2
		=
		a_n^2
		+
		o(a_n).
		\label{eq:zteps_squared}
	\end{equation}
	Since $z_{t_\varepsilon}>0$ and $a_n>0$,
	\begin{align}
		z_{t_\varepsilon}-a_n
		&=
		\frac{z_{t_\varepsilon}^2-a_n^2}
		{z_{t_\varepsilon}+a_n}
		\nonumber\\
		&=
		\frac{o(a_n)}{2a_n\{1+o(1)\}}
		=
		o(1).
		\label{eq:zteps_minus_an}
	\end{align}
	Thus
	\begin{equation}
		z_{t_\varepsilon}=a_n+o(1).
		\label{eq:zteps_asymp}
	\end{equation}
	
	Using \eqref{eq:zteps_asymp}, we have
	\[
	z_{t_\varepsilon}-a_n-b=-b+o(1),
	\]
	and
	\[
	z_{t_\varepsilon}+a_n+b=2a_n+b+o(1)\to\infty.
	\]
	Hence
	\[
	\bar\Phi(z_{t_\varepsilon}+a_n+b)=o(1).
	\]
	Therefore, by \eqref{eq:Gmu_representation},
	\begin{align}
		G_{a_n+b}(t_\varepsilon)
		&=
		\bar\Phi(z_{t_\varepsilon}-a_n-b)
		+
		\bar\Phi(z_{t_\varepsilon}+a_n+b)
		\nonumber\\
		&=
		\bar\Phi(-b+o(1))+o(1)
		\nonumber\\
		&=
		\Phi(b)+o(1)\nonumber\\
		&=
		q_b+o(1).
		\label{eq:endpoint_crossing}
	\end{align}
	It follows that
	\begin{equation}
		\frac{G_{a_n+b}(t_\varepsilon)}{q_\varepsilon}
		\to
		\frac{q_b}{q_b-\varepsilon}>1.
		\label{eq:endpoint_ratio_limit}
	\end{equation}
	Choose $\eta_\varepsilon>0$ such that
	\begin{equation}
		0<\eta_\varepsilon<
		\frac12
		\left\{
		\frac{q_b}{q_b-\varepsilon}-1
		\right\}.
		\label{eq:eta_choice}
	\end{equation}
	Then, by \eqref{eq:endpoint_ratio_limit}, for all sufficiently large $n$,
	\begin{equation}
		G_{a_n+b}(t_\varepsilon)
		\ge
		(1+\eta_\varepsilon)q_\varepsilon.
		\label{eq:endpoint_crossing_margin}
	\end{equation}
	
	Now fix $0<q\le q_\varepsilon$, and define
	\begin{equation}
		t_q=q\alpha_n\frac{s_n}{n}.
		\label{eq:tq_definition}
	\end{equation}
	Then $t_q\le t_\varepsilon$. By Lemma~\ref{lem:folded_normal_tail_ratio},
	the map $t\mapsto G_{a_n+b}(t)/t$ is nonincreasing on $(0,1)$.
	Therefore,
	\begin{equation}
		\frac{G_{a_n+b}(t_q)}{t_q}
		\ge
		\frac{G_{a_n+b}(t_\varepsilon)}{t_\varepsilon}.
		\label{eq:ratio_comparison_tq_teps}
	\end{equation}
	Using \eqref{eq:qeps_teps_def}, \eqref{eq:tq_definition}, and
	\eqref{eq:endpoint_crossing_margin}, we obtain
	\begin{align}
		G_{a_n+b}(t_q)
		&\ge
		\frac{t_q}{t_\varepsilon}
		G_{a_n+b}(t_\varepsilon)
		\nonumber\\
		&=
		\frac{q}{q_\varepsilon}
		G_{a_n+b}(t_\varepsilon)
		\nonumber\\
		&\ge
		(1+\eta_\varepsilon)q.
		\label{eq:crossing_for_all_q}
	\end{align}
	
	Finally, let $\theta\in\Theta_b$ and $i\in S_\theta$. Then
	$|\theta_i|\ge a_n+b$. Since $G_\mu(t)$ is symmetric in $\mu$ and
	nondecreasing as a function of $|\mu|$,
	\[
	G_{\theta_i}(t_q)\ge G_{a_n+b}(t_q).
	\]
	Combining this with \eqref{eq:crossing_for_all_q} gives
	\[
	G_{\theta_i}\!\left(q\alpha_n\frac{s_n}{n}\right)
	\ge
	(1+\eta_\varepsilon)q,
	\qquad i\in S_\theta,
	\]
	uniformly over $0<q\le q_b-\varepsilon$. This proves
	\eqref{eq:deterministic_crossing_conclusion}.
\end{proof}

The deterministic crossing established in the previous lemma must now be transferred to the empirical rejection path of the GBS procedure. The following result shows that this transfer occurs uniformly over the beta-min parameter space.

\begin{lemma}[Empirical signal crossing]
	\label{lem:empirical_crossing}
	Assume Lemmas~\ref{lem:deterministic_signal_crossing} and~\ref{lem:initial-crossing}.
	Fix $\eps\in(0,\Phi(b))$ and put
	\begin{equation}
		r_n=\left\lfloor \{\Phi(b)-\eps\}s_n\right\rfloor .
		\label{eq:rn_emp_crossing}
	\end{equation}
	Then
	\begin{equation}
		\sup_{\theta\in\Theta_b}
		\mathbb{P}_{\theta}(R_n^{\GBS}<r_n)\to0.
		\label{eq:emp_crossing_conclusion}
	\end{equation}
\end{lemma}

Lemma~\ref{lem:empirical_crossing} constitutes the final step in the signal-crossing analysis by transferring the deterministic crossing established in Lemma~\ref{lem:deterministic_signal_crossing} to the empirical GBS rejection path. Specifically, it shows that, uniformly over the beta-min class $\Theta_b$, the GBS procedure rejects at least $(\Phi(b)-\varepsilon)s_n$ hypotheses with probability tending to one. Together with the negligible false-positive contribution, this yields the corresponding upper bound on missed signals in expectation. Since the false-positive contribution is already known to be asymptotically negligible by Lemma~\ref{lem:false-positive}, this result yields the sharp characterization of the normalized Hamming risk and paves the way for the main sharp asymptotic minimaxity theorem.

\begin{proof}[{\bf Proof of Lemma~\ref{lem:empirical_crossing}}]
	Fix $\theta\in\Theta_b$. Recall that
	\begin{equation}
		N_1(t)=\sum_{i\in S_\theta}\mathbf 1\{P_i\le t\}.
		\label{eq:N1_def_emp_crossing}
	\end{equation}
If
\begin{equation}
	N_1(c_j)\ge j,\qquad j=1,\ldots,r_n,
	\label{eq:N1_crossing_all_j}
\end{equation}
then, for each $1\le j\le r_n$, at least $j$ signal $p$-values are no larger
than $c_j$. Consequently, at least $j$ of the full collection of $n$ $p$-values
are no larger than $c_j$, and hence
\begin{equation}
	P_{(j)}\le c_j,\qquad j=1,\ldots,r_n.
	\label{eq:Pj_leq_cj_from_N1}
\end{equation}
By the definition of the GBS step-down rejection number, this implies
$R_n^{\GBS}\ge r_n$. Therefore,
\begin{equation}
	\{R_n^{\GBS}<r_n\}
	\subseteq
	\left\{
	N_1(c_j)<j\ \text{for some }1\le j\le r_n
	\right\},
	\label{eq:emp_crossing_event_inclusion}
\end{equation}
and hence
\begin{equation}
	\mathbb P_\theta(R_n^{\GBS}<r_n)
	\le
	\mathbb P_\theta\left\{
	N_1(c_j)<j\ \text{for some }1\le j\le r_n
	\right\}.
	\label{eq:emp_crossing_reduction}
\end{equation}
	
	By Lemma~\ref{lem:initial-crossing},
	\begin{equation}
		m_n:=
		\inf_{\theta\in\Theta_b}\mathbb E_\theta N_1(c_1)
		\to\infty.
		\label{eq:mn_initial_crossing}
	\end{equation}
	Choose a sequence of positive integers $\{L_n\}$ such that
	\begin{equation}
		L_n\to\infty,\qquad
		L_n=o(m_n),\qquad
		L_n=o(s_n).
		\label{eq:Ln_conditions}
	\end{equation}
	For all sufficiently large $n$, $L_n\le m_n/2$. Since $c_j\ge c_1$,
	\begin{align}
		\mathbb P_\theta\left\{
		N_1(c_j)<j\ \text{for some }1\le j<L_n
		\right\}
		&\le
		\mathbb P_\theta\{N_1(c_1)<L_n\}
		\nonumber\\
		&\le
		\mathbb P_\theta\left\{
		N_1(c_1)<\frac12\mathbb E_\theta N_1(c_1)
		\right\}.
		\label{eq:early_rank_reduction}
	\end{align}
	By Chernoff's inequality,
	\begin{equation}
		\mathbb P_\theta\left\{
		N_1(c_1)<\frac12\mathbb E_\theta N_1(c_1)
		\right\}
		\le
		\exp\left\{-\frac18\mathbb E_\theta N_1(c_1)\right\}
		\le
		\exp\{-m_n/8\}.
		\label{eq:early_chernoff}
	\end{equation}
	Hence
	\begin{equation}
		\sup_{\theta\in\Theta_b}
		\mathbb P_\theta\left\{
		N_1(c_j)<j\ \text{for some }1\le j<L_n
		\right\}
		\to0.
		\label{eq:early_ranks_done}
	\end{equation}
	
	It remains to handle $L_n\le j\le r_n$. Put
	\begin{equation}
		q_j=\frac{j}{s_n}.
		\label{eq:qj_emp_crossing}
	\end{equation}
	By \eqref{eq:rn_emp_crossing},
	\begin{equation}
		0<q_j\le \Phi(b)-\eps,
		\qquad L_n\le j\le r_n.
		\label{eq:qj_range_emp_crossing}
	\end{equation}
	For all sufficiently large $n$, $\alpha_n<1/2$ and $L_n\ge2$. Thus,
	for $L_n\le j\le r_n$,
	\[
	j(1-\alpha_n)\ge1,
	\]
	and hence
	\begin{align}
		c_j
		=
		\frac{j\alpha_n}{n+1-j(1-\alpha_n)}
		\ge
		\frac{j\alpha_n}{n}
		=
		q_j\alpha_n\frac{s_n}{n}.
		\label{eq:cj_lower_bound_emp}
	\end{align}
	Since $G_{\theta_i}(t)$ is nondecreasing in $t$, Lemma~\ref{lem:deterministic_signal_crossing},
	\eqref{eq:qj_range_emp_crossing}, and \eqref{eq:cj_lower_bound_emp} imply that,
	for all sufficiently large $n$,
	\begin{equation}
		G_{\theta_i}(c_j)
		\ge
		(1+\eta_\eps)q_j,
		\qquad
		i\in S_\theta,\quad L_n\le j\le r_n.
		\label{eq:G_cj_lower_emp}
	\end{equation}
	Consequently,
	\begin{align}
		\mathbb E_\theta N_1(c_j)
		&=
		\sum_{i\in S_\theta}
		\mathbb P_{\theta_i}(P_i\le c_j)
		\nonumber\\
		&=
		\sum_{i\in S_\theta}
		G_{\theta_i}(c_j)
		\nonumber\\
		&\ge
		s_n(1+\eta_\eps)q_j
		\nonumber\\
		&= (1+\eta_\eps)j.
		\label{eq:mean_bulk_lower_emp}
	\end{align}
	By Chernoff's inequality,
	\begin{align}
		\mathbb P_\theta\{N_1(c_j)<j\}
		&\le
		\mathbb P_\theta\left\{
		N_1(c_j)
		<
		\frac{1}{1+\eta_\eps}
		\mathbb E_\theta N_1(c_j)
		\right\}
		\nonumber\\
		&\le
		\exp\left\{
		-\frac{\eta_\eps^2}{2(1+\eta_\eps)}j
		\right\}.
		\label{eq:bulk_chernoff_emp}
	\end{align}
	Let
	\begin{equation}
		C_\eps=\frac{\eta_\eps^2}{2(1+\eta_\eps)}>0.
		\label{eq:Ceps_emp}
	\end{equation}
	Then, by the union bound,
	\begin{align}
		\sup_{\theta\in\Theta_b}
		\mathbb P_\theta\left\{
		N_1(c_j)<j\ \text{for some }L_n\le j\le r_n
		\right\}
		&\le
		\sum_{j=L_n}^{r_n}\exp\{-C_\eps j\}
		\nonumber\\
		&\le
		\sum_{j=L_n}^{\infty}\exp\{-C_\eps j\}
		\to0.
		\label{eq:bulk_union_emp}
	\end{align}
	Combining \eqref{eq:emp_crossing_reduction}, \eqref{eq:early_ranks_done},
	and \eqref{eq:bulk_union_emp}, we obtain
	\[
	\sup_{\theta\in\Theta_b}
	\mathbb P_\theta(R_n^{\GBS}<r_n)\to0.
	\]
	This concludes the proof of Lemma~\ref{lem:empirical_crossing}.
\end{proof}

The preceding signal-crossing results provide precise control of the number of rejected signals. We now convert these probabilistic statements into the sharp asymptotic bound for the false-negative contribution.

\begin{lemma}[False negatives attain the sharp constant]
	\label{lem:false-negative}
	Under the assumptions of Theorem \ref{thm:hamming-main},
	\begin{equation}
		\sup_{\theta\in\Theta_b}
		\frac{1}{s_n}\mathbb{E}_{\theta} T(\theta,\varphi^{\GBS})
		\le \bar\Phi(b)+o(1).
		\label{eq:false_negative_bound}
	\end{equation}
\end{lemma}

Lemma~\ref{lem:false-negative} establishes that the false-negative contribution of the GBS procedure attains the sharp asymptotic minimax constant identified by \citet{ACR2024}. Combined with the asymptotically negligible false-positive contribution established in Lemma~\ref{lem:false-positive}, it follows that the normalized Hamming risk of the GBS procedure is asymptotically governed entirely by the universal constant $\bar{\Phi}(b)$. Consequently, all the essential ingredients required for proving sharp asymptotic minimaxity under normalized Hamming loss are now in place, and the main theorem follows by combining these two results.

\begin{proof}[{\bf Proof of Lemma~\ref{lem:false-negative}}]
	Fix $\varepsilon\in(0,\Phi(b))$ and let
	\[
	r_n=\left\lfloor \{\Phi(b)-\varepsilon\}s_n\right\rfloor .
	\]
	By Lemma~\ref{lem:empirical_crossing},
	\[
	\sup_{\theta\in\Theta_b}
	\mathbb P_\theta\{R_n^{\GBS}<r_n\}
	\to0.
	\]
	
	On the event $\{R_n^{\GBS}\ge r_n\}$, the GBS procedure makes at least
	$r_n$ total rejections. Since at most
	$V(\theta,\varphi^{\GBS})$ of these rejections can be false positives,
	the number of rejected signals is at least
	$r_n-V(\theta,\varphi^{\GBS})$. Hence, on this event,
	\[
	T(\theta,\varphi^{\GBS})
	\le
	s_n-r_n+V(\theta,\varphi^{\GBS}).
	\]
	On the complementary event, the trivial bound
	$T(\theta,\varphi^{\GBS})\le s_n$ holds. Therefore,
	\[
	T(\theta,\varphi^{\GBS})
	\le
	s_n-r_n
	+
	V(\theta,\varphi^{\GBS})
	+
	s_n\mathbf 1\{R_n^{\GBS}<r_n\}.
	\]
	Taking expectations, dividing by $s_n$, and then taking the supremum over
	$\theta\in\Theta_b$, we obtain
	\begin{align*}
		\sup_{\theta\in\Theta_b}
		\frac{1}{s_n}
		\mathbb E_\theta T(\theta,\varphi^{\GBS})
		&\le
		1-\frac{r_n}{s_n}
		+
		\sup_{\theta\in\Theta_b}
		\frac{1}{s_n}
		\mathbb E_\theta V(\theta,\varphi^{\GBS})
		+
		\sup_{\theta\in\Theta_b}
		\mathbb P_\theta\{R_n^{\GBS}<r_n\}.
	\end{align*}
	Now
	\[
	1-\frac{r_n}{s_n}
	\le
	\bar\Phi(b)+\varepsilon+\frac{1}{s_n}.
	\]
	Moreover, Lemma~\ref{lem:false-positive} gives
	\[
	\sup_{\theta\in\Theta_b}
	\frac{1}{s_n}
	\mathbb E_\theta V(\theta,\varphi^{\GBS})
	=o(1),
	\]
	and Lemma~\ref{lem:empirical_crossing} gives
	\[
	\sup_{\theta\in\Theta_b}
	\mathbb P_\theta\{R_n^{\GBS}<r_n\}
	=o(1).
	\]
	Therefore,
	\[
	\sup_{\theta\in\Theta_b}
	\frac{1}{s_n}
	\mathbb E_\theta T(\theta,\varphi^{\GBS})
	\le
	\bar\Phi(b)+\varepsilon+o(1).
	\]
	Since $\varepsilon\in(0,\Phi(b))$ is arbitrary, the desired bound follows.
\end{proof}

We now establish the support-size-uniform crossing extension required for Theorem~\ref{thm:at-most-hamming}. Its proof uses the preceding empirical crossing lemma along arbitrary sequences of support sizes.

\begin{lemma}[Support-size-uniform empirical crossing]
\label{lem:uniform_support_transfer}
Under Assumptions~\ref{ass:sparse}--\ref{ass:alpha}, fix
$\eta\in(0,1)$, $b\in\mathbb R$, $\zeta>0$, and
$\varepsilon\in(0,\Phi(b-\zeta))$. Put
$a_{n,k}=\sqrt{2\log(n/k)}$ and
$r_{n,k}=\lfloor(\Phi(b-\zeta)-\varepsilon)k\rfloor$.
Then
\[
\sup_{\eta s_n<k\le s_n}\ \sup_{\substack{\theta:\,\|\theta\|_0=k,\\
|\theta_i|\ge a_{n,k}+b-\zeta\ (i\in S_\theta)}}
\mathbb P_\theta(R_n^{\GBS}<r_{n,k})\longrightarrow0.
\]
\end{lemma}

\begin{proof}[{Proof of Lemma~\ref{lem:uniform_support_transfer}}]
If the uniform assertion failed, a subsequence $n_j$ and integers
$k_j\in(\eta s_{n_j},s_{n_j}]$ would give failure probabilities
bounded below by a positive constant. But $k_j\to\infty$,
$n_j/k_j\to\infty$, and
$\log(1/\alpha_{n_j})=o(\sqrt{2\log(n_j/k_j)})$.
Lemma~\ref{lem:empirical_crossing}, applied along this sparse
subsequence with sparsity $k_j$ and fixed offset $b-\zeta$,
forces these probabilities to vanish, a contradiction. The argument is sequence-uniform: the subsequence may have arbitrary variations in $k_j/s_{n_j}\in(\eta,1]$, and no limit of that ratio is required.
\end{proof}

We next record a monotonicity consequence of
Lemma~\ref{lem:folded_normal_tail_ratio}. This result will be used in the
heterogeneous signal-strength analysis to propagate signal-crossing bounds across
the GBS critical sequence.

\begin{lemma}[Tail-ratio monotonicity]
	\label{lem:heterogeneous_tail_ratio}
	For every $\mu\in\mathbb R$, the function
	\begin{equation}
		t\mapsto \frac{G_\mu(t)}{t}
		\label{eq:Gmu_over_t_monotone}
	\end{equation}
	is nonincreasing on $(0,1)$.
\end{lemma}

\begin{proof}[{\bf Proof of Lemma~\ref{lem:heterogeneous_tail_ratio}}]
	Since $G_\mu(t)=G_{-\mu}(t)$, it is enough to consider $\mu\ge0$.
	For $\mu=0$, we have $G_0(t)=t$, so the claim is immediate.
	Now suppose $\mu>0$. Let
	\[
	z_t=\Phi^{-1}(1-t/2),
	\qquad
	t=2\bar\Phi(z_t).
	\]
	Then
	\[
	G_\mu(t)
	=
	\bar\Phi(z_t-\mu)+\bar\Phi(z_t+\mu),
	\]
	and hence
	\[
	\frac{G_\mu(t)}{t}
	=
	\frac{\bar\Phi(z_t-\mu)+\bar\Phi(z_t+\mu)}
	{2\bar\Phi(z_t)}.
	\]
	By Lemma~\ref{lem:folded_normal_tail_ratio}, the right-hand side is
	nondecreasing as a function of $z_t$. Since $z_t$ is decreasing in $t$,
	it follows that $G_\mu(t)/t$ is nonincreasing in $t$.
\end{proof}

Lemma~\ref{lem:heterogeneous_tail_ratio} is a structural monotonicity result.
It allows lower bounds verified at one threshold to be propagated to smaller
thresholds along the GBS critical sequence. This property is used below in the
heterogeneous false-negative analysis.


We next extend the preceding analysis beyond the classical beta-min framework. Rather than assuming a common lower bound on all signal strengths, we allow the non-null means to vary across coordinates, following the heterogeneous signal-strength formulation introduced by \citet{ACR2024}. The next result shows that the empirical signal-crossing argument continues to hold uniformly in this more general setting.

\begin{lemma}[Initial crossing under heterogeneous signal strengths]
	\label{lem:heterogeneous_initial_crossing}
	Fix $\delta\in(0,1)$. Suppose that
	\begin{equation}
		1-\Lambda_n(a)\ge \delta
		\label{eq:qstar_lower_delta}
	\end{equation}
	along a subsequence. Then, uniformly over $\theta\in\Theta(a,s_n)$,
	\begin{equation}
		\mathbb E_\theta N_1(c_1)\to\infty,
		\label{eq:heterogeneous_initial_expectation}
	\end{equation}
	where
	\[
	N_1(t)=\sum_{i\in S_\theta}1\{P_i\le t\},
	\qquad
	c_1=\frac{\alpha_n}{n+\alpha_n}.
	\]
\end{lemma}

This result demonstrates that the deterministic and empirical signal-crossing methodology developed for the beta-min framework remains valid under heterogeneous signal strengths. The asymptotic false-negative contribution is now governed by the aggregate quantity $\Lambda_n(a)$, which reduces to $\bar\Phi(b)$ under the classical equal-signal-strength model.

\begin{proof}[{\bf Proof of Lemma~\ref{lem:heterogeneous_initial_crossing}}]
	Write
	\begin{equation}
		q_n
		=
		1-\Lambda_n(a)
		=
		\frac{1}{s_n}
		\sum_{j=1}^{s_n}\Phi(a_j-a_n^*).
		\label{eq:heterogeneous_qn_def}
	\end{equation}
	By assumption, $q_n\ge\delta$ along the subsequence under consideration.
	Choose $M>0$ so large that $\Phi(-M)\le\delta/2$, and define
	\[
	J_n=\{1\le j\le s_n: a_j\ge a_n^*-M\}.
	\]
	If $|J_n|/s_n<\rho_\delta$, where
	\[
	\rho_\delta=\frac{\delta}{2(1-\delta/2)},
	\]
	then
	\[
	q_n
	\le
	\frac{|J_n|}{s_n}
	+
	\left(1-\frac{|J_n|}{s_n}\right)\frac{\delta}{2}
	<\delta,
	\]
	which contradicts $q_n\ge\delta$. Hence
	\begin{equation}
		|J_n|\ge \rho_\delta s_n.
		\label{eq:heterogeneous_positive_fraction}
	\end{equation}
	
	Let
	\[
	z_{1,n}=\Phi^{-1}(1-c_1/2),
	\qquad
	c_1=\frac{\alpha_n}{n+\alpha_n}.
	\]
	For every signal satisfying $|\theta_i|\ge a_n^*-M$, the symmetry and
	monotonicity of $G_\mu(t)$ as a function of $|\mu|$ give
	\[
	G_{\theta_i}(c_1)
	\ge
	G_{a_n^*-M}(c_1)
	\ge
	\bar\Phi(z_{1,n}-a_n^*+M).
	\]
	Therefore, by \eqref{eq:heterogeneous_positive_fraction}, uniformly over
	$\theta\in\Theta(a,s_n)$,
	\begin{equation}
		\mathbb E_\theta N_1(c_1)
		=
		\sum_{i\in S_\theta}G_{\theta_i}(c_1)
		\ge
		\rho_\delta s_n\bar\Phi(z_{1,n}-a_n^*+M).
		\label{eq:heterogeneous_initial_lower}
	\end{equation}
	It remains to show that
	\begin{equation}
		s_n\bar\Phi(z_{1,n}-a_n^*+M)\to\infty.
		\label{eq:heterogeneous_tail_target}
	\end{equation}
	
	Put
	\[
	u_n=\log(n/s_n),
	\qquad
	v_n=\log s_n,
	\qquad
	A_n=\log(1/\alpha_n).
	\]
	Then $a_n^*=\sqrt{2u_n}$, $u_n\to\infty$, $v_n\to\infty$, and
	$A_n=o(\sqrt{u_n})$. Since
	\[
	\bar\Phi(z_{1,n})
	=
	\frac{c_1}{2}
	=
	\frac{\alpha_n}{2(n+\alpha_n)},
	\]
	the Gaussian tail bound $\bar\Phi(x)\le e^{-x^2/2}$ and
	$c_1/2\ge\alpha_n/(4n)$ for all sufficiently large $n$ imply
	\[
	z_{1,n}\le\sqrt{2\log(4n/\alpha_n)}.
	\]
	The factor $4$ is only a convenient fixed constant in this finite-sample
	comparison; it plays no role in the asymptotic conclusion. Put
	$C_0=\log 4$, so the bound is
	$z_{1,n}\le\sqrt{2(u_n+v_n+A_n+C_0)}$.
	Define
	\begin{equation}
		h_n=v_n+A_n+C_0,
		\qquad
		d_n=
		\sqrt{2(u_n+h_n)}-\sqrt{2u_n}+M.
		\label{eq:heterogeneous_hn_dn_def}
	\end{equation}
	Then
	\[
	z_{1,n}-a_n^*+M\le d_n,
	\]
	and hence
	\begin{equation}
		s_n\bar\Phi(z_{1,n}-a_n^*+M)
		\ge
		s_n\bar\Phi(d_n).
		\label{eq:heterogeneous_reduce_to_dn}
	\end{equation}
	
	We now prove that $s_n\bar\Phi(d_n)\to\infty$ without
	assuming that the ratio $h_n/u_n$ converges. Write
	\[
	\Delta_n=\sqrt{u_n+h_n}-\sqrt{u_n},
	\qquad d_n=\sqrt2\,\Delta_n+M,
	\qquad v_n=\log s_n.
	\]
	Since $h_n>0$ and $M>0$ was fixed above, $\Delta_n\ge0$ and
	$d_n\ge M>0$; in particular, the logarithm $\log(1+d_n)$
	and the one-sided Gaussian Mills lower bound used below are well defined.
	Since $h_n=v_n+A_n+C_0$ and $A_n=o(\sqrt{u_n})$,
	\begin{equation}
	v_n-\frac{d_n^2}{2}
	=2\sqrt{u_n}\,\Delta_n-\sqrt2 M\Delta_n
	-A_n-C_0-\frac{M^2}{2}.
	\label{eq:heterogeneous_exponent_identity_revised}
	\end{equation}
	We treat the two exhaustive ranges $v_n\le u_n$ and $v_n>u_n$;
	no limiting behavior of $v_n/u_n$ is required.
	If $v_n\le u_n$, then
	\[
	\Delta_n\le\frac{v_n+A_n+C_0}{2\sqrt{u_n}},
	\qquad
	\Delta_n^2\le\frac{(v_n+A_n+C_0)^2}{4u_n}
	\le\frac{v_n}{4}+o(v_n).
	\]
	Here the last bound follows from $v_n\le u_n$,
	$A_n=o(\sqrt{u_n})$, and $v_n\to\infty$.
	Moreover $\Delta_n=O(\sqrt{v_n})+o(1)$, so
	\begin{equation}
	v_n-\frac{d_n^2}{2}
	=v_n-\Delta_n^2-\sqrt2M\Delta_n-\frac{M^2}{2}
	\ge\frac{3v_n}{4}-o(v_n)-O(\sqrt{v_n})\to\infty.
	\label{eq:heterogeneous_exponent_case_small}
	\end{equation}
	In this range $d_n=O(\sqrt{v_n})+O(1)$, hence
	$\log(1+d_n)=O(\log(1+v_n))=o(v_n)$.
	If $v_n>u_n$, then $h_n\ge v_n>u_n$ and
	\[
	\Delta_n=\frac{h_n}{\sqrt{u_n+h_n}+\sqrt{u_n}}
	\ge(\sqrt2-1)\sqrt{u_n}.
	\]
	Consequently, by \eqref{eq:heterogeneous_exponent_identity_revised},
	\[
	v_n-\frac{d_n^2}{2}
	\ge \Delta_n(2\sqrt{u_n}-\sqrt2M)-A_n-O(1)
	\to\infty.
	\]
	More precisely, the right-hand side is bounded below by a
	positive constant times $\sqrt{u_n}\Delta_n$ for all large $n$,
	because $A_n=o(\sqrt{u_n})$ and
	$\Delta_n\ge(\sqrt2-1)\sqrt{u_n}$.
	Also $\log(1+d_n)=o(\sqrt{u_n}\Delta_n)$:
	indeed $\Delta_n\ge c\sqrt{u_n}\to\infty$ and
	$\log(1+\Delta_n)=o(\Delta_n)$.
	Thus in both ranges,
	\begin{equation}
	v_n-\frac{d_n^2}{2}-\log(1+d_n)\longrightarrow\infty.
	\label{eq:heterogeneous_main_exponent}
	\end{equation}
	By the standard Mills lower bound, for a universal constant $C>0$,
	\[
	\bar\Phi(d_n)\ge
	C\frac{\exp\{-d_n^2/2\}}{1+d_n}.
	\]
	Therefore
	\begin{equation}
	\log\{s_n\bar\Phi(d_n)\}
	\ge v_n-\frac{d_n^2}{2}-\log(1+d_n)+O(1)\to\infty.
	\label{eq:heterogeneous_mills_log}
	\end{equation}
	It follows that $s_n\bar\Phi(d_n)\to\infty$.
	Together with \eqref{eq:heterogeneous_reduce_to_dn}, this proves
	\eqref{eq:heterogeneous_tail_target}. Finally, combining
	\eqref{eq:heterogeneous_tail_target} with
	\eqref{eq:heterogeneous_initial_lower} gives
	\[
	\inf_{\theta\in\Theta(a,s_n)}
	\mathbb E_\theta N_1(c_1)\to\infty.
	\]
This concludes the proof of Lemma~\ref{lem:heterogeneous_initial_crossing}.	
\end{proof}

With the heterogeneous signal-crossing result in place, we are now in a position to derive the corresponding false-negative risk bound over the heterogeneous signal-strength classes. This result constitutes the principal ingredient needed for the heterogeneous sharp minimaxity theorems.

\begin{lemma}[Deterministic heterogeneous signal crossing]
	\label{lem:heterogeneous_deterministic_crossing}
	Fix $\delta\in(0,1)$. Suppose that
	\[
	q_n:=1-\Lambda_n(a)\ge \delta
	\]
	along a subsequence. Then there exist $\eta_\delta>0$ and $n_\delta\ge1$
	such that, for all sufficiently large $n$ belonging to that subsequence,
	\begin{equation}
		\inf_{\theta\in\Theta(a,s_n)}
		\frac{1}{s_n}
		\sum_{i\in S_\theta}
		G_{\theta_i}\left(
		q\alpha_n\frac{s_n}{n}
		\right)
		\ge
		(1+\eta_\delta)q
		\label{eq:heterogeneous_deterministic_crossing}
	\end{equation}
	uniformly over
	\begin{equation}
		0<q\le q_n-\delta.
		\label{eq:q_range_heterogeneous}
	\end{equation}
\end{lemma}

Lemma~\ref{lem:heterogeneous_deterministic_crossing} shows that the false-negative contribution continues to attain the optimal asymptotic benchmark over the heterogeneous parameter classes. Consequently, the extension of the sharp asymptotic minimaxity results follows by combining this bound with the asymptotically negligible false-positive contribution established earlier.

\begin{proof}[{\bf Proof of Lemma~\ref{lem:heterogeneous_deterministic_crossing}}]
	Set
	\begin{equation}
		q_{0,n}=q_n-\frac{\delta}{2},
		\qquad
		t_{0,n}=q_{0,n}\alpha_n\frac{s_n}{n}.
		\label{eq:q0n_t0n_def}
	\end{equation}
	Since $q_n\ge\delta$, we have $q_{0,n}\ge\delta/2$. Hence
	$\log(1/q_{0,n})=O(1)$. Therefore,
	\[
	\log(1/t_{0,n})
	=
	\log(n/s_n)+\log(1/\alpha_n)+O(1)
	=
	\frac{(a_n^*)^2}{2}+o(a_n^*).
	\]
	By the Gaussian quantile expansion, with
	$z_{0,n}=\Phi^{-1}(1-t_{0,n}/2)$,
	\begin{equation}
		z_{0,n}=a_n^*+o(1).
		\label{eq:z0_astar}
	\end{equation}
	
	Fix $\theta\in\Theta(a,s_n)$. By the definition of $\Theta(a,s_n)$,
	after relabeling the nonzero coordinates if necessary, we may write
	$S_\theta=\{i_1,\ldots,i_{s_n}\}$ with
	\[
	|\theta_{i_j}|\ge a_j,\qquad j=1,\ldots,s_n.
	\]
	Using the symmetry and monotonicity of $G_\mu(t)$ as a function of
	$|\mu|$,
	\[
	G_{\theta_{i_j}}(t_{0,n})
	\ge
	G_{a_j}(t_{0,n})
	\ge
	\bar\Phi(z_{0,n}-a_j).
	\]
	By \eqref{eq:z0_astar} and the Lipschitz continuity of $\Phi$,
	\begin{align}
		\frac{1}{s_n}
		\sum_{j=1}^{s_n}
		\bar\Phi(z_{0,n}-a_j)
		&=
		\frac{1}{s_n}
		\sum_{j=1}^{s_n}
		\Phi(a_j-z_{0,n})
		\nonumber\\
		&=
		\frac{1}{s_n}
		\sum_{j=1}^{s_n}
		\Phi(a_j-a_n^*)+o(1)
		\nonumber\\
		&=
		q_n+o(1).
		\label{eq:mean_at_t0_general}
	\end{align}
	Consequently, uniformly over $\theta\in\Theta(a,s_n)$,
	\begin{equation}
		\frac{1}{s_n}
		\sum_{i\in S_\theta}
		G_{\theta_i}(t_{0,n})
		\ge
		q_n-\frac{\delta}{8}
		\label{eq:mean_at_t0_lower}
	\end{equation}
	for all sufficiently large $n$.
	
	Now fix $0<q\le q_n-\delta$ and put
	\begin{equation}
		t_n(q)=q\alpha_n\frac{s_n}{n}.
		\label{eq:tnq_def}
	\end{equation}
	Since $q\le q_n-\delta<q_{0,n}$, we have $t_n(q)\le t_{0,n}$. By
	Lemma~\ref{lem:heterogeneous_tail_ratio},
	\[
	\frac{G_{\theta_i}(t_n(q))}{t_n(q)}
	\ge
	\frac{G_{\theta_i}(t_{0,n})}{t_{0,n}},
	\qquad i\in S_\theta.
	\]
	Averaging over $i\in S_\theta$ and using
	\eqref{eq:q0n_t0n_def}, \eqref{eq:mean_at_t0_lower}, and
	\eqref{eq:tnq_def}, we obtain
	\begin{align}
		\frac{1}{s_n}
		\sum_{i\in S_\theta}
		G_{\theta_i}(t_n(q))
		&\ge
		\frac{t_n(q)}{t_{0,n}}
		\frac{1}{s_n}
		\sum_{i\in S_\theta}
		G_{\theta_i}(t_{0,n})
		\nonumber\\
		&\ge
		q\frac{q_n-\delta/8}{q_n-\delta/2}.
		\label{eq:heterogeneous_crossing_propagation}
	\end{align}
	Since $q_n\le1$,
	\[
	\frac{q_n-\delta/8}{q_n-\delta/2}
	=
	1+\frac{3\delta/8}{q_n-\delta/2}
	\ge
	1+\frac{3\delta}{8}.
	\]
	Thus \eqref{eq:heterogeneous_deterministic_crossing} holds with, for
	example,
	\[
	\eta_\delta=\frac{3\delta}{16}.
	\]
	This completes the proof of Lemma~\ref{lem:heterogeneous_deterministic_crossing}.
\end{proof}

We now establish the empirical counterpart of the deterministic heterogeneous signal-crossing result obtained in Lemma~\ref{lem:heterogeneous_deterministic_crossing}. As in the beta-min setting, the deterministic crossing argument must be combined with concentration of the empirical signal counts in order to control the actual GBS rejection path. The following lemma shows that, uniformly over the heterogeneous signal-strength class, the empirical number of signal p-values below the GBS critical sequence closely follows its deterministic counterpart with overwhelming probability.

\begin{lemma}[Empirical heterogeneous signal crossing]
	\label{lem:heterogeneous_empirical_crossing}
	Fix $\delta\in(0,1)$ and define
	\[
	q_n=1-\Lambda_n(a).
	\]
	Let
	\begin{equation}
		r_n
		=
		\left\lfloor
		(q_n-\delta)s_n
		\right\rfloor .
		\label{eq:rn_heterogeneous}
	\end{equation}
	If $q_n\le\delta$, then $r_n\le0$. If $q_n>\delta$ along a subsequence, including when $(q_n-\delta)s_n$ stays bounded, then along that subsequence
	\begin{equation}
		\sup_{\theta\in\Theta(a,s_n)}
		\mathbb P_\theta
		\left(
		R_n^{\GBS}<r_n
		\right)
		\to0.
		\label{eq:heterogeneous_empirical_crossing_result}
	\end{equation}
\end{lemma}

Lemma~\ref{lem:heterogeneous_empirical_crossing} extends the empirical signal-crossing theory developed earlier for the beta-min framework to the heterogeneous signal-strength setting. Together with Lemma~\ref{lem:heterogeneous_deterministic_crossing}, it shows that the GBS rejection path continues to recover an asymptotically optimal proportion of the non-null hypotheses, even when the signal strengths vary across coordinates. Consequently, the deterministic crossing analysis remains valid after accounting for sampling fluctuations.

\begin{proof}[{\bf Proof of Lemma~\ref{lem:heterogeneous_empirical_crossing}}]
	If $q_n\le\delta$, then $r_n\le0$, and the assertion is immediate since
	$R_n^{\GBS}\ge0$. Hence work along an arbitrary subsequence on which $q_n>\delta$. The argument below does not require $r_n\to\infty$: if $r_n<L_n$, only the initial-rank estimate is used, whereas the later-rank union bound is taken over an empty index set.
	
	Fix $\theta\in\Theta(a,s_n)$ and recall that
	\[
	N_1(t)
	=
	\sum_{i\in S_\theta}\mathbf 1\{P_i\le t\}.
	\]
	If
	\[
	N_1(c_j)\ge j,\qquad j=1,\ldots,r_n,
	\]
	then $P_{(j)}\le c_j$ for every $1\le j\le r_n$, and hence
	$R_n^{\GBS}\ge r_n$. Therefore,
	\begin{equation}
		\mathbb P_\theta(R_n^{\GBS}<r_n)
		\le
		\mathbb P_\theta\left\{
		N_1(c_j)<j\ \text{for some }1\le j\le r_n
		\right\}.
		\label{eq:heterogeneous_empirical_reduction}
	\end{equation}
	
	By Lemma~\ref{lem:heterogeneous_initial_crossing},
	\[
	D_n:=
	\inf_{\theta\in\Theta(a,s_n)}
	\mathbb E_\theta N_1(c_1)
	\to\infty.
	\]
	Choose integers $L_n$ such that
	\[
	L_n\to\infty,\qquad
	L_n=o(D_n),\qquad
	L_n=o(s_n).
	\]
	For all sufficiently large $n$, $L_n\le D_n/2$. Since $c_j\ge c_1$ for
	all $j\ge1$,
	\begin{align}
		\mathbb P_\theta\left\{
		N_1(c_j)<j\ \text{for some }1\le j<L_n
		\right\}
		&\le
		\mathbb P_\theta\{N_1(c_1)<L_n\}
		\nonumber\\
		&\le
		\mathbb P_\theta\left\{
		N_1(c_1)<\frac12\mathbb E_\theta N_1(c_1)
		\right\}.
		\label{eq:heterogeneous_early_reduction}
	\end{align}
	By Chernoff's inequality,
	\[
	\mathbb P_\theta\left\{
	N_1(c_1)<\frac12\mathbb E_\theta N_1(c_1)
	\right\}
	\le
	\exp\left\{-\frac18\mathbb E_\theta N_1(c_1)\right\}
	\le
	\exp\{-D_n/8\}.
	\]
	Hence
	\begin{equation}
		\sup_{\theta\in\Theta(a,s_n)}
		\mathbb P_\theta\left\{
		N_1(c_j)<j\ \text{for some }1\le j<L_n
		\right\}
		\to0.
		\label{eq:early_heterogeneous}
	\end{equation}
	
	It remains to handle $L_n\le j\le r_n$. Put
	\[
	q_{j,n}=\frac{j}{s_n}.
	\]
	For $L_n\le j\le r_n$, we have
	\[
	0<q_{j,n}\le q_n-\delta.
	\]
	Moreover, for all sufficiently large $n$, $\alpha_n<1/2$ and
	$L_n\ge2$, so $j(1-\alpha_n)\ge1$. Therefore,
	\begin{align}
		c_j
		&=
		\frac{j\alpha_n}{n+1-j(1-\alpha_n)}
		\nonumber\\
		&\ge
		\frac{j\alpha_n}{n}
		=
		q_{j,n}\alpha_n\frac{s_n}{n}.
		\label{eq:heterogeneous_cj_lower}
	\end{align}
	By Lemma~\ref{lem:heterogeneous_deterministic_crossing} and the
	monotonicity of $G_\mu(t)$ in $t$,
	\begin{align}
		\inf_{\theta\in\Theta(a,s_n)}
		\mathbb E_\theta N_1(c_j)
		&=
		\inf_{\theta\in\Theta(a,s_n)}
		\sum_{i\in S_\theta}G_{\theta_i}(c_j)
		\nonumber\\
		&\ge
		s_n(1+\eta_\delta)q_{j,n}
		=
		(1+\eta_\delta)j,
		\label{eq:heterogeneous_bulk_mean_lower}
	\end{align}
	uniformly over $L_n\le j\le r_n$.
	
	For each such $j$, $N_1(c_j)$ is a sum of independent Bernoulli random
	variables. Hence Chernoff's inequality gives
	\begin{align}
		\mathbb P_\theta\{N_1(c_j)<j\}
		&\le
		\mathbb P_\theta\left\{
		N_1(c_j)
		<
		\frac{1}{1+\eta_\delta}
		\mathbb E_\theta N_1(c_j)
		\right\}
		\nonumber\\
		&\le
		\exp\left\{
		-\frac{\eta_\delta^2}{2(1+\eta_\delta)}j
		\right\}.
		\label{eq:heterogeneous_bulk_chernoff}
	\end{align}
	Let
	\[
	C_\delta=\frac{\eta_\delta^2}{2(1+\eta_\delta)}>0.
	\]
	By the union bound,
	\begin{align}
		\sup_{\theta\in\Theta(a,s_n)}
		\mathbb P_\theta\left\{
		N_1(c_j)<j\ \text{for some }L_n\le j\le r_n
		\right\}
		&\le
		\sum_{j=L_n}^{r_n}\exp\{-C_\delta j\}
		\nonumber\\
		&\le
		\sum_{j=L_n}^{\infty}\exp\{-C_\delta j\}
		\to0.
		\label{eq:bulk_heterogeneous}
	\end{align}
	
	Combining \eqref{eq:heterogeneous_empirical_reduction},
	\eqref{eq:early_heterogeneous}, and \eqref{eq:bulk_heterogeneous}
	proves
	\[
	\sup_{\theta\in\Theta(a,s_n)}
	\mathbb P_\theta(R_n^{\GBS}<r_n)\to0.
	\]
This completes the proof of Lemma~\ref{lem:heterogeneous_empirical_crossing}.	
\end{proof}

Having established the empirical signal-crossing property for heterogeneous signal strengths, we are now in a position to derive the corresponding false-negative risk bound. This result plays the same role as Lemma 8 in the beta-min framework and provides the principal ingredient required for proving sharp asymptotic minimaxity over the general parameter class $\Theta(a,s_n)$.

\begin{lemma}[False-negative bound under heterogeneous signal-strength classes]
	\label{lem:false_negative_general}
	Under the sparsity condition \eqref{eq:general_sparse_regime} and Assumption~\ref{ass:alpha},
	\begin{equation}
		\sup_{\theta\in\Theta(a,s_n)}
		\frac{1}{s_n}
		\mathbb E_\theta
		T(\theta,\varphi^{\GBS})
		\le
		\Lambda_n(a)+o(1).
		\label{eq:false_negative_general_bound}
	\end{equation}
\end{lemma}

Lemma~\ref{lem:false_negative_general} provides the sharp asymptotic characterization of the false-negative contribution over the heterogeneous signal-strength classes. The leading term is governed by the aggregate quantity $\Lambda_n(a)$ introduced by \citet{ACR2024}, thereby recovering the exact sharp minimax benchmark for general signal strengths. Combined with the asymptotically negligible false-positive contribution established earlier, this lemma immediately yields the sharp asymptotic minimaxity results stated in Theorems 4 and 5.

\begin{proof}[{\bf Proof of Lemma~\ref{lem:false_negative_general}}]
	Fix $\delta\in(0,1)$ and write
	\begin{equation}
		q_n=1-\Lambda_n(a).
		\label{eq:qn_false_negative_general}
	\end{equation}
	If $q_n\le\delta$, then, for every $\theta\in\Theta(a,s_n)$,
	\begin{equation}
		\frac{1}{s_n}T(\theta,\varphi^{\GBS})
		\le
		1
		\le
		\Lambda_n(a)+\delta.
		\label{eq:false_negative_general_small_q}
	\end{equation}
	
	Now suppose $q_n>\delta$, and define
	\begin{equation}
		r_n=\lfloor(q_n-\delta)s_n\rfloor.
		\label{eq:rn_false_negative_general}
	\end{equation}
	By Lemma~\ref{lem:heterogeneous_empirical_crossing},
	\begin{equation}
		\sup_{\theta\in\Theta(a,s_n)}
		\mathbb P_\theta(R_n^{\GBS}<r_n)
		\to0.
		\label{eq:false_negative_general_emp_crossing}
	\end{equation}
	On the event $\{R_n^{\GBS}\ge r_n\}$,
	\[
	T(\theta,\varphi^{\GBS})
	\le
	s_n-r_n+V(\theta,\varphi^{\GBS}).
	\]
	On the complementary event, $T(\theta,\varphi^{\GBS})\le s_n$. Hence
	\begin{equation}
		T(\theta,\varphi^{\GBS})
		\le
		s_n-r_n
		+
		V(\theta,\varphi^{\GBS})
		+
		s_n\mathbf 1\{R_n^{\GBS}<r_n\}.
		\label{eq:false_negative_general_decomposition}
	\end{equation}
	Taking expectations, dividing by $s_n$, and taking the supremum gives
	\begin{align}
		\sup_{\theta\in\Theta(a,s_n)}
		\frac{1}{s_n}
		\mathbb E_\theta T(\theta,\varphi^{\GBS})
		&\le
		1-\frac{r_n}{s_n}
		+
		\sup_{\theta\in\Theta(a,s_n)}
		\frac{1}{s_n}
		\mathbb E_\theta V(\theta,\varphi^{\GBS})
		\nonumber\\
		&\qquad+
		\sup_{\theta\in\Theta(a,s_n)}
		\mathbb P_\theta(R_n^{\GBS}<r_n).
		\label{eq:false_negative_general_expectation_bound}
	\end{align}
	By \eqref{eq:rn_false_negative_general}, we obtain
	\begin{equation}
		1-\frac{r_n}{s_n}
		\le
		1-(q_n-\delta)+\frac{1}{s_n}
		=
		\Lambda_n(a)+\delta+\frac{1}{s_n}.
		\label{eq:false_negative_general_floor_bound}
	\end{equation}
	Moreover,
	\begin{equation}
		\sup_{\theta\in\Theta(a,s_n)}
		\frac{1}{s_n}
		\mathbb E_\theta V(\theta,\varphi^{\GBS})
		=o(1)
		\label{eq:false_positive_general_negligible}
	\end{equation}
	by Lemma~\ref{lem:false-positive}, and
	\begin{equation}
		\sup_{\theta\in\Theta(a,s_n)}
		\mathbb P_\theta(R_n^{\GBS}<r_n)
		=o(1)
		\label{eq:false_negative_general_rejection_failure}
	\end{equation}
	by Lemma~\ref{lem:heterogeneous_empirical_crossing}. Combining
	\eqref{eq:false_negative_general_expectation_bound}--\eqref{eq:false_negative_general_rejection_failure},
	we obtain
	\begin{equation}
		\sup_{\theta\in\Theta(a,s_n)}
		\frac{1}{s_n}
		\mathbb E_\theta T(\theta,\varphi^{\GBS})
		\le
		\Lambda_n(a)+\delta+o(1).
		\label{eq:false_negative_general_delta_bound}
	\end{equation}
	The bounds above hold on both index sets $\{n:q_n\le\delta\}$ and $\{n:q_n>\delta\}$, with a remainder uniform over the latter. Thus they combine into a single bound along the full sequence, even if $q_n$ oscillates or approaches $0$ or $1$. Since $\delta\in(0,1)$ is arbitrary, the desired bound follows.
\end{proof}

\bibliography{GBS_Sharp_Minimaxity}

%
%
%
%
%

\end{document}